\documentclass[12pt]{article}
\usepackage{algorithm}
\usepackage{algpseudocode}
\usepackage{amsfonts}
\usepackage{amsmath}
\usepackage{amssymb} 
\usepackage{amsthm} 
\usepackage[title]{appendix} 
\usepackage{autobreak}
\usepackage{bigdelim} 
\usepackage{blkarray}
\makeatletter
\def\BA@fnsymbol#1{\ensuremath{%
  \ifcase#1\or *\or \dagger\or \ddagger\or \mathsection\or \mathparagraph\or \|\or **\or \dagger\dagger\or \ddagger\ddagger \else@ctrerr\fi}}%
\makeatother
\usepackage{bm}
\usepackage{booktabs}
\usepackage{cancel} 
\usepackage{cases}
\usepackage{chngcntr}
\usepackage{colortbl} 
\usepackage{comment}
\usepackage{empheq} 
\usepackage{enumerate}
\usepackage{fancyhdr} 
\usepackage{fancyvrb} 
\usepackage{framed} 
\usepackage{graphicx}
\usepackage{hyperref} 
\hypersetup{
  colorlinks,
  linkcolor={red!50!black},
  citecolor={blue!50!black},
  urlcolor={blue!80!black}
}
\usepackage{mathrsfs} 
\usepackage{multirow} 
\usepackage[sort&compress,numbers]{natbib}
\usepackage{setspace} 
\usepackage{stmaryrd}
\usepackage{subfig} 
\usepackage{pdflscape} 
\usepackage{pgfplots} 
\usepackage{pxrubrica} 
\usepackage{url}
\usepackage[colorinlistoftodos,prependcaption,textsize=footnotesize]{todonotes}

\usepackage[top=35truemm,bottom=35truemm,left=30truemm,right=30truemm]{geometry} 
\makeatletter

\@addtoreset{equation}{section}
\makeatother

\algtext*{EndWhile}
\algtext*{EndIf}
\algtext*{EndFor}

\theoremstyle{plain}
\newtheorem{theorem}{Theorem}[section]

\newtheorem{lemma}[theorem]{Lemma}

\theoremstyle{definition}

\theoremstyle{remark}
\newtheorem{remark}[theorem]{Remark}%
\newtheorem{question}[theorem]{Question}%

\DeclareMathOperator*{\setspan}{span}

\DeclareMathOperator*{\tr}{tr}

\newcommand{\relmiddle}[1]{\mathrel{}\middle#1\mathrel{}}

\newcommand{\norm}[1]{\lVert #1 \rVert} 

\newcommand{\bbE}{\mathbb{E}}

\newcommand{\bbK}{\mathbb{K}}

\newcommand{\bbR}{\mathbb{R}}
\newcommand{\bbS}{\mathbb{S}}

\newcommand{\bbU}{\mathbb{U}}
\newcommand{\bbV}{\mathbb{V}}

\newcommand{\calI}{\mathcal{I}}
\newcommand{\calJ}{\mathcal{J}}

\newcommand{\calS}{\mathcal{S}}

\newcommand{\calV}{\mathcal{V}}

\newcommand{\CP}{\mathcal{CP}}
\newcommand{\COP}{\mathcal{COP}}

\makeatletter
\newcommand{\dotpplus}{\mathbin{\text{\@dotpplus}}}
\newcommand{\@dotpplus}{%
  \ooalign{\hidewidth\hbox{$\circ$}\hidewidth\cr$\m@th+$\cr}%
}
\makeatother
\newcommand{\plmod}{\ensuremath{\dotpplus}}
\makeatletter
\newcommand{\dotminus}{\mathbin{\text{\@dotminus}}}
\newcommand{\@dotminus}{%
  \ooalign{\hidewidth\hbox{$\circ$}\hidewidth\cr$\m@th-$\cr}%
}
\makeatother
\newcommand{\mnmod}{\ensuremath{\dotminus}}

\newcommand{\RNum}[1]{\uppercase\expandafter{\romannumeral #1\relax}} 
\newcommand{\Rnum}[1]{\lowercase\expandafter{\romannumeral #1\relax}} 

\title{Copositive and completely positive cones over symmetric cones of rank at least 5}

\makeatletter
\let\@fnsymbol\@arabic
\makeatother

\author{
\normalsize
    Mitsuhiro Nishijima\thanks{Department of Industrial and Systems Engineering, Keio University, 3-14-1 Hiyoshi, Kohoku-ku, Yokohama-shi, 2238522, Kanagawa, Japan. ({\tt nishijima@keio.jp}).}
}

\begin{document}
\maketitle

\begin{abstract}\noindent
We study copositive and completely positive cones over symmetric cones of rank at least $5$, with particular emphasis on whether these cones are spectrahedral shadows and on the behavior of a sum-of-squares inner-approximation hierarchy.
We examine to what extent known results for nonnegative orthants of dimension at least $5$ carry over to general symmetric cones of rank at least $5$.
We first prove that neither the copositive nor the completely positive cone over such a symmetric cone is a spectrahedral shadow.
We then generalize the Horn matrix to this setting by introducing Horn transformations and analyzing their geometric and algebraic properties.
We show that Horn transformations generate exposed rays of copositive cones over symmetric cones and that they evade the zeroth level of the sum-of-squares inner-approximation hierarchy.
Finally, we examine the asymptotic exactness of this hierarchy over positive semidefinite cones.
In contrast to the $5$-dimensional nonnegative orthant, where the hierarchy is known to recover the entire copositive cone in the limit, we construct instances over positive semidefinite cones of order at least $5$ certifying that the union of all levels remains strictly included in the copositive cone.
\end{abstract}
\vspace{0.5cm}


\section{Introduction}\label{sec:intro}
Let $\bbK$ be a closed cone in a finite-dimensional real vector space $\bbV$ with inner product $\bullet$.
A self-adjoint linear transformation $A$ on $\bbV$ is said to be \emph{copositive over $\bbK$} if $x\bullet A(x) \ge 0$ for all $x \in \bbK$, and \emph{completely positive over $\bbK$} if there exist a positive integer $m$ and vectors $a_1,\dots,a_m \in \bbK$ such that $A = \sum_{i=1}^m a_i\otimes a_i$.
We refer to the cones of self-adjoint linear transformations that are copositive over $\bbK$ and completely positive over $\bbK$ as the \emph{copositive cone over $\bbK$} and the \emph{completely positive cone over $\bbK$}, respectively.

In this paper, we focus on the case where $\bbK$ is a \emph{symmetric cone}, that is, a self-dual and homogeneous cone.
Typical examples of symmetric cones include nonnegative orthants, second-order cones, the cone of real symmetric positive semidefinite matrices, and their direct products.
If $\bbK$ is a nonnegative orthant, copositivity and complete positivity reduce to the standard notions~\cite{SB2021}.
Using copositive and completely positive cones over symmetric cones, many NP-hard optimization problems can be reformulated in a unified way as conic linear programs~\cite{Burer2009,Burer2012,BMP2016,BD2012,BDd+2000}.

We study whether copositive and completely positive cones over symmetric cones are \emph{spectrahedral shadows}.
Roughly speaking, a spectrahedral shadow is the feasible set of a semidefinite programming problem, for which polynomial-time algorithms are available~\cite{WSV2000}.
Formally, as recalled in Section~\ref{subsec:spec}, it is defined as a projection of a \emph{spectrahedron}.

There are existing theoretical results on whether copositive and completely positive cones over symmetric cones are spectrahedral shadows.
First, Bodirsky, Kummer, and Thom~\cite[Corollary~3.18]{BKT2026} showed that a copositive cone over a nonnegative orthant of dimension $n\ge 5$ is not a spectrahedral shadow.
This result implies that its dual, a completely positive cone over a nonnegative orthant of dimension $n \ge 5$, is not a spectrahedral shadow either.
Conversely, it follows from \cite[Theorem~1]{Diananda1962} and \cite{MM1962} that copositive and completely positive cones over a nonnegative orthant of dimension $n\le 4$ are spectrahedral shadows.
Second, by the S-lemma~\cite[Theorem~2.2]{PT2007}, copositive and completely positive cones over a second-order cone are spectrahedral shadows for any dimension.
For further details on spectrahedral shadow representations for these cones, see Theorem~1 and Corollary~5 of \cite{SZ2003}.
Third, the result shown by Nishijima and Louren\c{c}o~\cite{NL2025} implies that copositive cones over symmetric cones of dimension at least $2$ are not spectrahedra.
However, such copositive cones may be spectrahedral shadows, as in the examples of a nonnegative orthant of dimension at most $4$ and of a second-order cone.

Numerically, copositive cones over symmetric cones can be well approximated by spectrahedral shadows.
Nishijima and Nakata~\cite{NN2024_Approximation} proposed a hierarchical inner approximation of copositive cones over symmetric cones, where each level of the hierarchy is a spectrahedral shadow.
When the symmetric cone is a nonnegative orthant, the hierarchy reduces to that proposed by Parrilo~\cite{Parrilo2000_Semidefinite}.
In the experiments reported in \cite{NN2024_Approximation,NN2024_Generalizations}, the optimal values of various problems with copositive constraints were numerically identical to the approximate optimal values obtained using even the zeroth level of the hierarchy.

For specific symmetric cones, the agreement or lack thereof between the
copositive cone over the symmetric cone and the zeroth level of this inner-approximation hierarchy has been clarified.
On the one hand, it was shown in \cite[Proposition~4.12]{NN2024_Generalizations} that a copositive cone over a second-order cone always coincides with the zeroth level of the hierarchy.
On the other hand, for a nonnegative orthant, the copositive cone coincides with the zeroth level of the hierarchy if and only if the dimension of the nonnegative orthant is at most $4$, which follows from \cite[Section~4]{Parrilo2000_Semidefinite} and \cite{Diananda1962}.
In particular, in dimensions at most $4$, it also coincides with the union of all levels of the hierarchy.
When the dimension of the nonnegative orthant is at least $5$, the Horn matrix
\begin{equation}
\bm{H} \coloneqq \begin{pmatrix}
1 & -1 & 1 & 1 & -1\\
-1&  1 &-1 & 1 &  1\\
 1& -1 & 1 &-1 &  1\\
 1&  1 &-1 & 1 & -1\\
-1&  1 & 1 &-1 &  1
\end{pmatrix} \label{eq:Horn_mat}
\end{equation}
or its embedding, i.e., the symmetric matrix obtained by padding $\bm{H}$ with
zeros, provides a certificate that the copositive cone over the nonnegative
orthant does not coincide with the zeroth level of the hierarchy
\cite[Section~5]{Parrilo2000_Semidefinite}.

For nonnegative orthants, the asymptotic exactness of the inner-approximation hierarchy is completely understood in terms of the dimension.
The union of all levels coincides with the copositive cone over the nonnegative orthant when the dimension is at most $5$, by the preceding zeroth-level result for dimensions at most $4$ and \cite[Theorem~3]{SV2024} for dimension $5$.
In contrast, when the dimension is at least $6$, the union is strictly included in the copositive cone~\cite[page~1024]{LV2023}.

The contributions of this paper are threefold.
First, we extend the result of \cite{BKT2026}, which was established for nonnegative orthants, to general symmetric cones.
Specifically, we prove that for any symmetric cone of \emph{rank} at least $5$, the associated copositive and completely positive cones are not spectrahedral shadows; see Theorem~\ref{thm:COP_not_shadow}.
Here, the rank of a symmetric cone plays the same role as the dimension of a nonnegative orthant, although in general it does not coincide with the dimension of the symmetric cone.
As discussed in Section~\ref{sec:remark}, when the rank of a symmetric cone is $1$ or $2$, the associated copositive and completely positive cones are spectrahedral shadows.
Interestingly, the case where the rank is $3$ or $4$ is still open, and we do not know whether the associated copositive and completely positive cones are spectrahedral shadows.

Second, we generalize the Horn matrix defined in \eqref{eq:Horn_mat} to the setting of symmetric cones of rank at least $5$ by introducing \emph{Horn transformations}.
We show in Theorem~\ref{thm:Horn_exposed_symcone} that the Horn transformations generate exposed rays of copositive cones over symmetric cones.
This parallels the fact, supported by \cite[page~335]{HN1963} and \cite[Theorem~4.6.\Rnum{3}]{Dickinson2011}, that the Horn matrix generates an exposed ray of the copositive cone over the $5$-dimensional nonnegative orthant.

Third, we demonstrate the non-exactness of the inner-approximation hierarchy proposed by Nishijima and Nakata~\cite{NN2024_Approximation} when the rank of the symmetric cone is at least $5$.
We show in Theorem~\ref{thm:not_sos} that the Horn transformations do not belong to the zeroth level of the inner-approximation hierarchy.
As the inner-approximation hierarchy uses cones that are spectrahedral shadows, a copositive cone and its inner approximation cannot coincide in view of Theorem~\ref{thm:COP_not_shadow}.
The Horn transformations provide concrete examples of elements that belong to the copositive cone but not to the zeroth level of the corresponding inner-approximation hierarchy.
Moreover, for positive semidefinite cones of order at least $5$, we show in Theorem~\ref{thm:not_exact_psd} that the associated copositive cone strictly includes the union of all levels of this hierarchy.
This contrasts with the case of the nonnegative orthant of dimension $5$, where the associated copositive cone coincides with the union of all levels of the hierarchy.

We emphasize the significance of our result that Horn transformations generate exposed rays of copositive cones over symmetric cones.
For a convex cone, exposed rays are precisely $1$-dimensional exposed faces.
Faces and their exposedness play a central role in deriving error bounds for associated conic linear feasibility problems~\cite{Lourenco2021,LLP2023,LLL+2024,LLL+2025,LLP2025,WLP20XX}.
Recently, the faces and their exposedness for copositive and completely positive cones over symmetric cones have been studied in \cite{NL2025,NL2026}.
From this perspective, the present work can be positioned within this line of research.
Furthermore, for the nonnegative orthant of dimension $5$, the Horn matrix is a key ingredient in the classification of the extreme rays of the associated copositive cone~\cite{Hildebrand2012}.
Although the classification of the extreme rays of copositive cones over symmetric cones of rank $5$ remains an open problem, the results of this paper may contribute to the development of analogous classifications in this broader setting.

This paper is organized as follows.
In Section~\ref{sec:preliminaries}, we present the notation and preliminary technical results required later.
In Section~\ref{sec:non_spectrahedral_shadow}, we show that copositive and completely positive cones over symmetric cones of rank at least $5$ are not spectrahedral shadows.
In Section~\ref{sec:Horn}, we introduce Horn transformations and show that they generate exposed rays of copositive cones over symmetric cones.
In Section~\ref{sec:non_exact}, we demonstrate the non-exactness of the inner-approximation hierarchy over both general and particular symmetric cones.
In Section~\ref{sec:remark}, we conclude with final remarks.

\section{Preliminaries}\label{sec:preliminaries}
\subsection{Basic Notation and Terminology}
Let $V$ denote a finite-dimensional real vector space endowed with an inner product.
For a subset $S$ of $V$, we write $S^\perp$ and $\setspan S$ for the set of elements orthogonal to every element in $S$ and the minimal subspace containing $S$, respectively.
We use $\bbR$ to denote the field of real numbers.
For $x\in V$, we define $\bbR x \coloneqq \{ax \mid a\in \bbR\}$ and $\bbR_+ x \coloneqq \{ax \mid a\ge 0\}$.

We call a subset $K$ of $V$ a \emph{cone} if $ax \in K$ for all $a \ge 0$ and $x\in K$.
A linear transformation $f$ on $V$ is called an \emph{automorphism} of a closed convex cone $K$ if $f$ is invertible and satisfies $f(K) = K$.
We say that closed convex cones $K_1$ and $K_2$ are \emph{linearly isomorphic} if one can find a linear isomorphism $f\colon \setspan K_1 \to \setspan K_2$ satisfying $f(K_1) = K_2$.
The \emph{dual cone} of a subset $S$ of $V$ is the cone of $x\in V$ such that the inner product between $x$ and $y$ is nonnegative for all $y\in S$.
For a closed convex cone $K$, a nonzero element $x\in K$ is said to generate an \emph{extreme ray} $\bbR_+x$ of $K$ if, whenever $a,b\in K$ satisfy $a+b\in \bbR_+x$, it follows that $a,b\in \bbR_+x$.
Moreover, a nonzero element $x\in K$ is said to generate an \emph{exposed ray} $\bbR_+x$ of $K$ if there exists an element $d$ in the dual cone of $K$ such that $\bbR_+x = K \cap \{d\}^\perp$.
Note that every exposed ray is an extreme ray.

Vectors are written in bold lowercase letters, e.g., $\bm{a}$.
The $i$th element of a vector $\bm{a}$ is written as $a_i$.
The vector whose $i$th element is $1$ and whose remaining elements are $0$ is denoted by $\bm{e}_i$.
For vectors $\bm{a}$ and $\bm{b}$, we write $\bm{a}\odot \bm{b}$ for the entrywise product of $\bm{a}$ and $\bm{b}$.
We use $\bbR^n$ to denote the space of $n$-tuples of real numbers.
We endow $\bbR^n$ with the standard inner product.
$\bbR_+^n$, an $n$-dimensional nonnegative orthant, is defined as the set of nonnegative vectors in $\bbR^n$.

Matrices are written in bold uppercase letters, e.g., $\bm{A}$.
The $(i,j)$th element of a matrix $\bm{A}$ is written as $A_{ij}$.
For a square matrix $\bm{A}$, we use $\tr(\bm{A})$ to denote the trace of $\bm{A}$.
We use $\bbS^n$ to denote the space of real $n\times n$ symmetric matrices.
Furthermore, we write $\bbS_+^n$ for the cone of positive semidefinite matrices in $\bbS^n$, which we call the positive semidefinite cone.

\subsection{Symmetric Cones and Euclidean Jordan Algebras}\label{subsec:EJA}
This subsection provides preliminaries on symmetric cones and Euclidean Jordan algebras.
As will be seen later, there is a one-to-one correspondence between these two objects, and we frequently treat symmetric cones through the framework of Euclidean Jordan algebras.
Hence we first introduce Euclidean Jordan algebras and then turn to symmetric cones.

Let $\bbE$ be a finite-dimensional real vector space equipped with a bilinear product $\circ\colon \bbE\times \bbE\to \bbE$.
We write $x\circ x$ as $x^2$ for every $x\in \bbE$.
We call $(\bbE,\circ)$ a \emph{Jordan algebra} if the product $\circ$ is commutative and satisfies $x\circ (x^2\circ y) = x^2 \circ (x\circ y)$ for all $x,y\in \bbE$.
In this paper, we assume that every Jordan algebra has a unit element.
In addition, a Jordan algebra $(\bbE,\circ)$ is called \emph{Euclidean} if we can equip $\bbE$ with an \emph{associative} inner product $\bullet\colon \bbE\times \bbE \to \bbR$, where the associativity means that $(x\circ y)\bullet z = x\bullet (y\circ z)$ for all $x,y,z\in \bbE$.
In this paper, when $(\bbE,\circ)$ is a Euclidean Jordan algebra, we fix an associative inner product and write the algebra as $(\bbE,\circ,\bullet)$.
If no confusion can arise, we abbreviate $(\bbE,\circ,\bullet)$ to $\bbE$.
The norm induced by this inner product is denoted by $\norm{\cdot}$.

Let $(\bbE,\circ,\bullet)$ be a Euclidean Jordan algebra.
An element $c\in \bbE$ satisfying $c^2 = c$ is called an \emph{idempotent}.
An idempotent $c \in \bbE$ is referred to as \emph{primitive} if it is nonzero and cannot be decomposed into a sum of two nonzero idempotents in $\bbE$.
Idempotents $c_1,\dots,c_k \in \bbE$ are said to be \emph{orthogonal} if $i\neq j$ implies $c_i \circ c_j = 0$.
A \emph{Jordan frame} of $\bbE$ is then a set $\{c_1,\dots,c_r\}$ of orthogonal primitive idempotents in $\bbE$ such that $\sum_{i=1}^r c_i$ is the unit element of $\bbE$.
The number $r$ of elements in a Jordan frame of $\bbE$ is independent of the choice of the frame~\cite[Section~\RNum{3}.1]{FK1994}, and we call it the \emph{rank} of the Euclidean Jordan algebra $\bbE$.

For an idempotent $c\in \bbE$ and $\lambda \in \bbR$, we define $\bbE(c,\lambda) \coloneqq \{x\in \bbE \mid c\circ x = \lambda x\}$.
We note that $\bbE(c,1)$ is a Euclidean Jordan subalgebra of $\bbE$~\cite[Proposition~\RNum{4}.1.1]{FK1994}.
In addition, for a Jordan frame $\{c_1,\dots,c_r\}$ of $\bbE$, we introduce
\begin{alignat*}{3}
\bbE_{ii} &\coloneqq \bbE(c_i,1) = \bbR c_i &\ &(i = 1,\dots,r),\\
\bbE_{ij} &\coloneqq \textstyle \bbE(c_i,\frac{1}{2}) \cap \bbE(c_j,\frac{1}{2}) && (i,j = 1,\dots,r,\ i\neq j).
\end{alignat*}
These subspaces yield the following orthogonal direct sum decomposition~\cite[Theorem~\RNum{4}.2.1.\Rnum{3}]{FK1994}:
\begin{equation}
\bbE = \bigoplus_{1\le i\le j\le r}\bbE_{ij}. \label{eq:Peirce}
\end{equation}
This decomposition is called the \emph{Peirce decomposition} of $\bbE$ with respect to the Jordan frame.
It follows from \cite[Lemma~20]{GST2004} that
\begin{equation}
\bbE\left(\sum_{l=1}^k c_l,1\right) = \bigoplus_{1\le i\le j\le k}\bbE_{ij}. \label{eq:Peirce_truncated}
\end{equation}
Furthermore, for every $1\le i < j \le r$ and $x \in \bbE_{ij}$, the following two equalities hold:
\begin{align}
x^2 &= \underbrace{c_i\circ x^2}_{\in \bbE_{ii}} + \underbrace{c_j \circ x^2}_{\in \bbE_{jj}}, \label{eq:x^2}\\
c_i\bullet (c_i \circ x^2) &= \frac{1}{2}\norm{x}^2, \label{eq:ci_ci_xij2}
\end{align}
where \eqref{eq:x^2} follows from \cite[Proposition~\RNum{4}.1.1]{FK1994} and its proof, and \eqref{eq:ci_ci_xij2} follows from the associativity of the inner product $\bullet$ and the definition of $\bbE_{ij}$.

A closed cone $\bbK$ in a finite-dimensional real inner product space is said to be a \emph{symmetric cone} if it is self-dual and homogeneous.
The self-duality means that the dual cone of $\bbK$ is equal to $\bbK$, and the homogeneity means that for any interior points $x,y$ of $\bbK$, one can choose an automorphism $G$ of $\bbK$ with $G(x) = y$.

There is a one-to-one correspondence between symmetric cones and Euclidean Jordan algebras.
For every Euclidean Jordan algebra $\bbE$, the set $\bbE_+ \coloneqq \{x^2 \mid x\in \bbE\}$ forms a symmetric cone~\cite[Theorem~\RNum{3}.2.1]{FK1994}.
Conversely, for every symmetric cone $\bbK$ in a finite-dimensional real inner product space $(\bbE,\bullet)$, one can define a bilinear product $\circ\colon \bbE\times\bbE \to \bbE$ such that $(\bbE,\circ,\bullet)$ is a Euclidean Jordan algebra and satisfies $\bbE_+ = \bbK$~\cite[Theorem~\RNum{3}.3.1]{FK1994}.
Moreover, such a Euclidean Jordan algebra $\bbE$ is unique up to isomorphism in a Jordan algebraic sense.\footnote{The uniqueness follows from the classification of simple Euclidean Jordan algebras~\cite[Chapter~\RNum{5}]{FK1994}.
See also \cite[Section~2.1]{Chua2009} for further discussion from the viewpoints of homogeneous cones and $T$-algebras.}
We define the rank of the symmetric cone $\bbK$ as that of the associated Euclidean Jordan algebra $\bbE$.

There is a one-to-one correspondence between extreme rays of a symmetric cone $\bbK$ and primitive idempotents of its associated Euclidean Jordan algebra $\bbE$~\cite[Proposition~\RNum{4}.3.2]{FK1994}.
Using this correspondence, we see that for any positive integer $k$ not exceeding the rank of $\bbK$, there is a one-to-one correspondence between families of $k$ mutually orthogonal extreme rays of $\bbK$ and sets of $k$ orthogonal primitive idempotents in $\bbE$.
In particular, let $\delta_1,\dots,\delta_k \in \bbK$ be mutually orthogonal elements that generate extreme rays of $\bbK$.
Since each $\delta_i$ generates an extreme ray of $\bbK$, it is a positive multiple of some primitive idempotent $c_i$ in $\bbE$.
Consequently, $\{c_1,\dots,c_k\}$ forms a set of orthogonal primitive idempotents in $\bbE$.

The nonnegative orthant $\bbR_+^n$ is an example of a symmetric cone.
Its associated Euclidean Jordan algebra is the \emph{Hadamard Euclidean Jordan algebra} $(\bbR^n,\circ,\bullet)$, where the bilinear product $\circ$ and the inner product $\bullet$ are defined by
\begin{equation*}
\bm{x}\circ \bm{y} \coloneqq \bm{x}\odot \bm{y}
\quad \text{and} \quad
\bm{x}\bullet \bm{y} \coloneqq \bm{x}^\top \bm{y},
\end{equation*}
respectively, for $\bm{x},\bm{y}\in \bbR^n$.
The Jordan frame of this Euclidean Jordan algebra is equal to the standard basis $\{\bm{e}_1,\dots,\bm{e}_n\}$ for $\bbR^n$.
This implies that the rank of the Hadamard Euclidean Jordan algebra $\bbR^n$ is $n$, and so is that of the nonnegative orthant $\bbR_+^n$.

\subsection{Linear Mappings}
We assume throughout this subsection that $(\bbV,\bullet)$ is a finite-dimensional real inner product space.
For a subspace $\bbU$ of $\bbV$, we write $P_{\bbU}\colon \bbV \to \bbU$ for the orthogonal projection associated with the direct sum $\bbV = \bbU \oplus \bbU^\perp$.
For $a,b\in \bbV$, the tensor product of $a$ and $b$ is written as $a\otimes b$.
For simplicity, we write $a^{\otimes 2}$ for the tensor product of $a$ with itself.
Under the natural isomorphism between the tensor product of $\bbV$ with itself and the space of linear transformations on $\bbV$, the tensor product $a\otimes b$ can be regarded as the linear transformation $x\mapsto (b\bullet x)a$.
The space of self-adjoint linear transformations on $\bbV$ is denoted by $\calS(\bbV)$.
We use $\langle \cdot,\cdot\rangle$ to denote the trace inner product on $\calS(\bbV)$.
For $A \in \calS(\bbV)$, we define $q_A \colon \bbV \to \bbR$ by $q_A(x) \coloneqq x\bullet A(x)$ for $x\in \bbV$.

We assume in the following that $\bbV$ can be written as the orthogonal direct sum
\begin{equation}
\bbV = \bigoplus_{i=1}^k\bbV_i. \label{eq:V_orth_decomp}
\end{equation}
For every $A\in \calS(\bbV)$ and $i,j = 1,\dots,k$, we define $A_{i,j}\colon \bbV_j \to \bbV_i$ (or $A_{ij}$ for short) as $A_{ij} \coloneqq P_{\bbV_i}A|_{\bbV_j}$.
As explained in \cite[Section~2.2]{NL2025}, $A_{ij}$ may be viewed as the ``$(i,j)$th element'' of $A$ in analogy with symmetric matrices.
For example, for every $x\in \bbV$, we can write, in accordance with \eqref{eq:V_orth_decomp}, $x = \sum_{i=1}^k x_i$ with $x_i \in \bbV_i$.
Then it follows that
\begin{equation*}
x\bullet A(x) = \sum_{i=1}^k x_i \bullet A_{ii}(x_i) + 2\sum_{1\le i<j\le k}x_i\bullet A_{ij}(x_j).
\end{equation*}
Accordingly, we can write $A$ in the following matrix-like representation:
\begin{equation*}
A = (A_{ij})_{1\le i\le j\le k} = \bordermatrix{
& 1 & 2 & \cdots & k \cr
& A_{11} & A_{12} & \cdots & A_{1k} \cr
&  & A_{22} & \cdots & A_{2k} \cr
&  &  & \ddots & \vdots \cr
&  &  &   & A_{kk}}.
\end{equation*}

The following lemma collects basic facts about self-adjoint linear transformations.
The proof is straightforward and is omitted.
\begin{lemma}\label{lem:submatrix_zero}
Let $A,B\in \calS(\bbV)$.
\begin{enumerate}[(i)]
\item For every $i,j = 1,\dots,k$, we have $A_{ij} = 0$ if and only if $x_i \bullet A_{ij}(x_j) = 0$ for all $x_i\in \bbV_i$ and $x_j\in \bbV_j$. \label{enum:element_0}
\item For each $j = 1,2$ and a subset $I_j$ of $\{1,\dots,k\}$, we define $\bbU_j \coloneqq \bigoplus_{i\in I_j}\bbV_i$.
Then $P_{\bbU_1}A|_{\bbU_2} = P_{\bbU_1}B|_{\bbU_2}$ if and only if $A_{ij} = B_{ij}$ for all $(i,j)\in I_1\times I_2$.
In particular, by the symmetry of $A$ and $B$, we see that $P_{\bbU_1}A|_{\bbU_1} = P_{\bbU_1}B|_{\bbU_1}$ if and only if $A_{ij} = B_{ij}$ for all $(i,j)\in I_1\times I_1$ such that $i\le j$.
Consequently, $A = B$ if and only if $A_{ij} = B_{ij}$ for all $1\le i\le j\le k$.
\label{enum:element_equal}
\end{enumerate}
\end{lemma}

Using matrix terminology, we interpret \eqref{enum:element_equal} of Lemma~\ref{lem:submatrix_zero} as follows: the submatrices extracted from two symmetric matrices by the same rows and columns are identical if and only if the corresponding elements of these submatrices are identical.
Thus, for $A\in \calS(\bbV)$, we refer to the mapping of the form $P_{\bbU}A|_{\bbU} \in \calS(\bbU)$ for some subspace $\bbU$ of $\bbV$ as a \emph{principal subtransformation} of $A$.

\subsection{Copositive and Completely Positive Cones}
Let $(\bbV,\bullet)$ be a finite-dimensional real inner product space and $\bbK$ be a closed cone in $\bbV$.
We recall the definitions of copositivity and complete positivity over $\bbK$ introduced in Section~\ref{sec:intro}.
We define the cones of mappings $A\in \calS(\bbV)$ that are copositive over $\bbK$ and completely positive over $\bbK$ as $\COP(\bbK)$ and $\CP(\bbK)$, respectively.
When we emphasize the domain (and the range) $\bbV$ of elements in $\COP(\bbK)$, we use $\COP(\bbK;\bbV)$ instead of $\COP(\bbK)$.
In other words, $\COP(\bbK;\bbU)$ is a cone in $\calS(\bbU)$ for a subspace $\bbU$ of $\bbV$ with $\bbK \subseteq \bbU$, and $\COP(\bbK;\bbU)$ is not equal to $\COP(\bbK;\bbV)$ in general.
Note that if $\bbU$ is a subspace of $\bbV$ with $\bbK \subseteq \bbU$, then every element in $\bbK$ lies in both $\bbV$ and $\bbU$, so it is unnecessary to indicate whether $\CP(\bbK)$ is a subset of $\calS(\bbV)$ or $\calS(\bbU)$.
For that reason, we do not use notation like $\CP(\bbK;\bbV)$.
It is known that $\COP(\bbK;\bbV)$ and $\CP(\bbK)$ are dual to each other~\cite[Section~2]{SZ2003}.
When $\bbV = \bbR^n$ and $\bbK = \bbR_+^n$, under the linear isomorphism between $\calS(\bbR^n)$ and $\bbS^n$, the sets $\COP(\bbR_+^n)$ and $\CP(\bbR_+^n)$ can be viewed as subsets of $\bbS^n$, which we write as $\COP^n$ and $\CP^n$, respectively.

The following lemma corresponds to the fact that a principal submatrix of a copositive matrix is also copositive~\cite[Proposition~2.4]{SB2021}.
The proof follows immediately from the definitions, so we omit it.
\begin{lemma}\label{lem:subtf_cop}
Let $\bbV'$ be a subspace of $\bbV$ and $\bbK'$ be a closed cone included in $\bbK \cap \bbV'$.
For every $A \in \COP(\bbK;\bbV)$, we have $P_{\bbV'}A|_{\bbV'} \in \COP(\bbK';\bbV')$.
\end{lemma}

We can also see the next lemma by analogy with matrices: for $\bm{A} \in \bbS^n$ and $\bm{B}\in \CP^k$ with $k\le n$, the inner product of $\bm{A}$ with the symmetric matrix obtained by padding $\bm{B}$ with zeros equals that of the corresponding principal submatrix of $\bm{A}$ with $\bm{B}$.
\begin{lemma}\label{lem:A_dot_B_reduced}
Let $\bbU$ be a subspace of $\bbV$.
For every $A \in \calS(\bbV)$ and $B \in \CP(\bbU)$, we have $\langle A,B\rangle = \langle P_{\bbU}A|_{\bbU},B\rangle$, where the inner product that appears in the left-hand side is defined on the space $\calS(\bbV)$ and that in the right-hand side is defined on the space $\calS(\bbU)$.
\end{lemma}

\subsection{Spectrahedra and Their Shadows}\label{subsec:spec}
Let $X$ be a finite-dimensional real vector space.
A subset $S$ of $X$ is called a \textit{spectrahedron} if there exist a positive integer $d$, a matrix $\bm{M}_0\in \bbS^d$, and a linear mapping $\bm{M}\colon X \to \bbS^d$ such that
\begin{equation*}
S = \{x \in X \mid \bm{M}_0 + \bm{M}(x) \in \bbS_+^d\}.
\end{equation*}
A subset $C$ of $X$ is called a \textit{spectrahedral shadow}\footnote{In the literature, spectrahedral shadows are also referred to as \emph{semidefinite representable} sets~\cite{BN2001} and as \emph{projected spectrahedra}~\cite{BPT2012}, for example.} if it is a projection of a spectrahedron, i.e., there exist a finite-dimensional real vector space $Y$, a positive integer $d$, a matrix $\bm{M}_0\in \bbS^d$, and a linear mapping $\bm{M}\colon X\times Y \to \bbS^d$ such that
\begin{equation*}
C = \{x \in X \mid \text{there exists $y\in Y$ such that }\bm{M}_0 + \bm{M}(x,y)\in \bbS_+^d\}.
\end{equation*}
Every spectrahedron is, by definition, a spectrahedral shadow, but the converse does not hold in general.

The class of spectrahedral shadows is closed under intersection, inverse images under linear mappings, and duality~\cite[Theorem~3.5]{NP2023}.
The following lemma gathers these properties for later use.
\begin{lemma}\label{lem:spec_shadow}
Let $X_1,X_2$ be finite-dimensional real vector spaces.
\begin{enumerate}[(i)]
\item If $C_1$ and $C_2$ are spectrahedral shadows in $X_1$, then $C_1 \cap C_2$ is also a spectrahedral shadow. \label{enum:intersection}
\item Let $f\colon X_1 \to X_2$ be a linear mapping.
If $C$ is a spectrahedral shadow in $X_2$, then the inverse image $f^{-1}(C)$ is a spectrahedral shadow in $X_1$.\label{enum:image}
\item We further assume that $X_1$ is equipped with an inner product.
If $C$ is a spectrahedral shadow in $X_1$, then the dual cone of $C$ is also a spectrahedral shadow.\label{enum:dual}
\end{enumerate}
\end{lemma}

\subsection{Polynomials}
Let $\bbV$ be an $n$-dimensional real vector space equipped with an inner product denoted by $(\cdot,\cdot)$.
In general, every element in the symmetric algebra on the dual space of $\bbV$ is regarded as a polynomial (function) with real coefficients on $\bbV$~\cite[Section~\RNum{5}.4.2]{Satake1975}.
Let $\{v_1,\dots,v_n\}$ be a basis for $\bbV$.
The linear mapping $x\mapsto (x,\cdot)$ is a natural isomorphism from $\bbV$ to its dual space.
Then a homogeneous polynomial (or simply, a form) $f(x)$ of degree $m$ with real coefficients on $\bbV$ can be written as
\begin{equation*}
f(x) = \sum_{i_1,\dots,i_m=1}^n f_{i_1\cdots i_m}(v_{i_1},x)\cdots (v_{i_m},x),
\end{equation*}
where each $f_{i_1\cdots i_m}$ is a real number.
The definition of homogeneous polynomials in the author's previous work~\cite{NN2024_Approximation} is stated in a different form, but it is consistent with the present definition.
A polynomial $f$ on $\bbV$ is said to be a \emph{sum of squares} if it can be written as a sum of squares of polynomials, and is said to be \emph{positive semidefinite} if $f(x) \ge 0$ for all $x\in \bbV$.
We write $\Sigma^{n,2m}(\bbV)$ for the cone of sums of squares of forms of degree $m$ with real coefficients on $\bbV$.

\section{Nonexistence of Spectrahedral Shadow Representations for Copositive and Completely Positive Cones over Symmetric Cones}\label{sec:non_spectrahedral_shadow}
Let $\bbK$ be a symmetric cone of rank $r$ in a finite-dimensional real inner product space $(\bbE,\bullet)$.
In this section, we show that neither $\COP(\bbK)$ nor $\CP(\bbK)$ is a spectrahedral shadow if $r \ge 5$.
We prove this by taking a slice of $\COP(\bbK)$ that is linearly isomorphic to $\COP^r$.

For $\delta_1,\dots,\delta_r \in \bbK$ that are orthogonal to each other and generate extreme rays of $\bbK$, we define
\begin{equation}
\calV \coloneqq \left\{\sum_{i,j=1}^r M_{ij}\delta_i\otimes \delta_j \relmiddle| \bm{M} \in \bbS^r\right\}, \label{eq:calV}
\end{equation}
which is a subspace of $\calS(\bbE)$.
We note that the linear mapping $f\colon \bbS^r \to \calV$ defined as $f(\bm{M}) \coloneqq \sum_{i,j=1}^r M_{ij}\delta_i\otimes \delta_j$ is a linear isomorphism,
since the vectors $\delta_i^{\otimes 2}$ for $i = 1,\dots,r$ and $\delta_i\otimes \delta_j +  \delta_j\otimes \delta_i$ for $1\le i< j\le r$ are linearly independent.

\begin{lemma}\label{lem:COPK_slice}
It follows that $\COP(\bbK) \cap \calV = f(\COP^r)$.
In particular, $\COP(\bbK) \cap \calV$ is linearly isomorphic to $\COP^r$.
\end{lemma}

\begin{proof}
Let $A \in f(\COP^r)$.
Then there exists $\bm{M} \in \COP^r$ such that $A = \sum_{i,j=1}^r M_{ij}\delta_i\otimes \delta_j$.
It follows from the definition of $\calV$ that $A \in \calV$.
Let $x\in \bbK$ and set $y_i \coloneqq \delta_i \bullet x$ for each $i = 1,\dots,r$.
Since every $\delta_i$ belongs to $\bbK$ and $\bbK$ is self-dual, the vector $\bm{y} \coloneqq (y_1,\dots,y_r)^\top$ belongs to $\bbR_+^r$.
Therefore, it follows from $\bm{M} \in \COP^r$ that
\begin{equation*}
q_A(x) = \bm{y}^\top\bm{M}\bm{y} \ge 0.
\end{equation*}
Since $x\in \bbK$ is arbitrary, we have $A\in \COP(\bbK)$.

Conversely, let $A \in \COP(\bbK) \cap \calV$.
By $A \in \calV$, there exists $\bm{M}\in \bbS^r$ such that $A = \sum_{i,j=1}^r M_{ij}\delta_i\otimes \delta_j$.
For any $\bm{y} \in \bbR_+^r$, we define
\begin{equation*}
x \coloneqq \sum_{i=1}^r\frac{y_i}{\delta_i\bullet \delta_i}\delta_i \in \bbK.
\end{equation*}
Then it follows from $A \in \COP(\bbK)$ that
\begin{align*}
0&\le q_A(x)\\
&= \sum_{i,j=1}^r\frac{y_iy_j}{(\delta_i\bullet \delta_i)(\delta_j\bullet \delta_j)}\delta_i\bullet \left(\sum_{k,l=1}^r M_{kl}\delta_k\otimes \delta_l\right)(\delta_j) \\
&= \sum_{i,j=1}^rM_{ij}y_iy_j\\
&= \bm{y}^\top\bm{M}\bm{y},
\end{align*}
where the second equality follows from the orthogonality of $\delta_1,\dots,\delta_r$.
This indicates $\bm{M} \in \COP^r$, so we see that $A \in f(\COP^r)$.
\end{proof}

\begin{theorem}\label{thm:COP_not_shadow}
Suppose that the rank $r$ of the symmetric cone $\bbK$ is at least $5$.
Then neither $\COP(\bbK)$ nor $\CP(\bbK)$ is a spectrahedral shadow.
\end{theorem}

\begin{proof}
We suppose that $\COP(\bbK)$ is a spectrahedral shadow.
The subspace $\calV$ is polyhedral, so it is a spectrahedral shadow~\cite[page~38]{RG1995}.
Therefore, by \eqref{enum:intersection} of Lemma~\ref{lem:spec_shadow}, the intersection $\COP(\bbK) \cap \calV$ is also a spectrahedral shadow.
Since $\COP(\bbK) \cap \calV$ is linearly isomorphic to $\COP^r$ by Lemma~\ref{lem:COPK_slice}, $\COP^r$ is also a spectrahedral shadow.
However, $\COP^r$ is not a spectrahedral shadow~\cite[Corollary~3.18]{BKT2026}, which is a contradiction.
Thus, $\COP(\bbK)$ is not a spectrahedral shadow.
By the duality between $\COP(\bbK)$ and $\CP(\bbK)$, it follows from \eqref{enum:dual} of Lemma~\ref{lem:spec_shadow} that $\CP(\bbK)$ is not a spectrahedral shadow either.
\end{proof}

In fact, we can directly show that $\CP(\bbK)$ is not a spectrahedral shadow if $r\ge 5$.
For the linear subspace $\calV$ defined in \eqref{eq:calV}, it follows that $\CP(\bbK) \cap \calV = f(\CP^r)$, which is linearly isomorphic to $\CP^r$.
Thus, in a way similar to Theorem~\ref{thm:COP_not_shadow}, we see that $\CP(\bbK)$ is not a spectrahedral shadow.

\section{Horn Transformations and Their Exposedness}\label{sec:Horn}
Throughout this section, let $\bbK$ denote a symmetric cone of rank at least $5$ in a finite-dimensional inner product space $(\bbE,\bullet)$.
Let $(\bbE,\circ,\bullet)$ be an associated Euclidean Jordan algebra satisfying $\bbE_+ = \bbK$.
For a set $\Delta = \{\delta_1,\dots,\delta_5\}$ whose elements are pairwise orthogonal and generate extreme rays of $\bbK$, we define
\begin{equation}
H_{\Delta} \coloneqq \sum_{i=1}^5 \delta_i^{\otimes 2} - \sum_{i=1}^5(\delta_i \otimes \delta_{i\plmod 1} + \delta_{i\plmod 1}\otimes \delta_i) + \sum_{i=1}^5(\delta_i \otimes \delta_{i\plmod 2} + \delta_{i\plmod 2}\otimes \delta_i), \label{eq:Horn_tr}
\end{equation}
where $\plmod$ denotes the addition modulo $5$ on $\{1,\dots,5\}$ (e.g., $4 \plmod 2 = 1$).
We call $H_{\Delta}$ the \emph{Horn transformation}.

As the name indicates, when the symmetric cone $\bbK$ is the nonnegative orthant $\bbR_+^5$, the Horn transformation reduces to the Horn matrix $\bm{H}$ introduced in \eqref{eq:Horn_mat}.
The extreme rays of $\bbR_+^5$ are generated by $\bm{e}_1,\dots,\bm{e}_5$.
Under the linear isomorphism between $\calS(\bbR^5)$ and $\bbS^5$, we have
\begin{equation*}
H_{\{\bm{e}_1,\dots,\bm{e}_5\}} = \sum_{i=1}^5 \bm{e}_i\bm{e}_i^\top - \sum_{i=1}^5(\bm{e}_i\bm{e}_{i\plmod 1}^\top + \bm{e}_{i\plmod 1}\bm{e}_i^\top)
+ \sum_{i=1}^5(\bm{e}_i\bm{e}_{i\plmod 2}^\top + \bm{e}_{i\plmod 2}\bm{e}_i^\top) = \bm{H}.
\end{equation*}

The goal of this section is to show that the Horn transformation $H_{\Delta}$ generates an exposed ray of $\COP(\bbK)$.
This result extends the fact that the Horn matrix $\bm{H}$ generates an exposed ray of $\COP^5$, which is a consequence of \cite[page~335]{HN1963} and \cite[Theorem~4.6.\Rnum{3}]{Dickinson2011}.

We provide a brief outline of the proof.
First, we establish the result within the framework of Euclidean Jordan algebras.
For a set $C = \{c_1,\dots,c_5\}$ of orthogonal primitive idempotents in the Euclidean Jordan algebra $\bbE$, we first show that $\bbR_+ H_C$ is an exposed ray of $\COP(\bbK)$ when $\bbE$ has rank~$5$, as proved in Lemma~\ref{lem:Horn_exposed_rank5}, and then extend this result to arbitrary rank in Lemma~\ref{lem:Horn_exposed}.
Finally, using Lemma~\ref{lem:Horn_exposed}, we prove Theorem~\ref{thm:Horn_exposed_symcone}, which is the main result of this section.

Before proceeding further, we establish the following two basic lemmas on Horn transformations.

\begin{lemma}\label{lem:Horn_not_psd}
For a set $\Delta = \{\delta_1,\dots,\delta_5\}$ whose elements are pairwise orthogonal and generate extreme rays of $\bbK$, the Horn transformation $H_{\Delta}$ is not positive semidefinite.
\end{lemma}

\begin{proof}
We let $x \coloneqq -\delta_1/\norm{\delta_1}^2 - \delta_2/\norm{\delta_2}^2 + \delta_4/\norm{\delta_4}^2 \in \bbE$.
Then we have $q_{H_{\Delta}}(x) = -3$.
Therefore, $H_{\Delta}$ is not positive semidefinite.
\end{proof}

\begin{lemma}\label{lem:Horn_cop}
For a set $\Delta = \{\delta_1,\dots,\delta_5\}$ whose elements are pairwise orthogonal and generate extreme rays of $\bbK$, we have $H_{\Delta} \in \COP(\bbK)$.
\end{lemma}

\begin{proof}
Let $x\in \bbK$.
Then the vector
\begin{equation}
\bm{v}_{\Delta}(x) \coloneqq (\delta_1\bullet x,\delta_2\bullet x,\delta_3\bullet x,\delta_4\bullet x,\delta_5\bullet x)^\top \label{eq:vdelta}
\end{equation}
belongs to $\bbR_+^5$.
It follows from $\bm{H} \in \COP^5$ that
\begin{equation*}
q_{H_{\Delta}}(x) = \bm{v}_{\Delta}(x)^\top\bm{H}\bm{v}_{\Delta}(x) \ge 0. 
\end{equation*}
Thus, we obtain $H_{\Delta}\in \COP(\bbK)$.
\end{proof}

In the subsequent proofs, we use a matrix representation of elements in $\calS(\bbE)$ associated with a Peirce decomposition.
Let $r$ denote the rank of $\bbE$ and $\{c_1,\dots,c_r\}$ be a Jordan frame of $\bbE$.
As shown in \eqref{eq:Peirce}, we can obtain the Peirce decomposition of $\bbE$ with respect to the Jordan frame.
We define a lexicographic-type order $\preceq$ on the set $\{(i,j) \mid 1\le i \le j\}$ by comparing the second component first and then the first; that is, $(i,j) \preceq (k,l)$ if and only if either $j < l$, or $j = l $ and $i\le k$.
For notational convenience, we write $(ij,kl)$ instead of $((i,j),(k,l))$ when no ambiguity can arise.
For a positive integer $m$, we define
\begin{equation*}
\llbracket m\rrbracket \coloneqq \{(ij,kl) \mid 1\le i\le j\le m,\ 1\le k\le l\le m,\ 11 \preceq ij \preceq kl \preceq mm\}.
\end{equation*}
Then every $A\in \calS(\bbE)$ can be represented as $A = (A_{ij,kl})_{(ij,kl)\in \llbracket r\rrbracket}$, where $A_{ij,kl} = P_{\bbE_{ij}} A|_{\bbE_{kl}}$ is the $(ij,kl)$th element of $A$.

In accordance with the matrix representation, for a set $C = \{c_1,\dots,c_5\}$ of orthogonal primitive idempotents in $\bbE$, by complementing the set $C$ so that $\{c_1,\dots,c_r\}$ is a Jordan frame of $\bbE$ if necessary, we can represent $H_C$ as
\begin{equation*}
H_C = \bordermatrix{
& 11 & 22 & 33 & 44 & 55 \cr
& c_1^{\otimes 2} & -c_1\otimes c_2 & c_1\otimes c_3 & c_1\otimes c_4 & -c_1\otimes c_5 \cr
&  & c_2^{\otimes 2} & -c_2\otimes c_3 & c_2\otimes c_4 & c_2\otimes c_5 \cr
&  &  & c_3^{\otimes 2} & -c_3\otimes c_4 & c_3\otimes c_5 \cr
&  &  &  & c_4^{\otimes 2} & -c_4\otimes c_5 \cr
&  &  &  & & c_5^{\otimes 2}}.
\end{equation*}
Note that such complements $c_6,\dots,c_r$ to $C$ always exist; see \cite[page~6]{NL2025}.
\begin{lemma}\label{lem:Horn_exposed_rank5}
Suppose that the rank of the Euclidean Jordan algebra $\bbE$ is $5$.
For a Jordan frame $C = \{c_1,\dots,c_5\}$ of $\bbE$, the Horn transformation $H_C$ generates an exposed ray of $\COP(\bbE_+)$.
\end{lemma}

\begin{proof}
We consider the Peirce decomposition~\eqref{eq:Peirce} of $\bbE$ with respect to the Jordan frame $C$.
For simplicity, write $H$ for $H_C$.
In this proof, we introduce the following notation.
For $(i,j) $ with $1\le j < i \le 5$, we define
\begin{equation}
(i,j) \coloneqq (j,i). \label{eq:ij_notation}
\end{equation}
This notation is justified since $\bbE_{ij} = \bbE_{ji}$ holds by definition.
In addition, for $(kl,ij)\in \llbracket 5\rrbracket$ with $kl \neq ij$, we define
\begin{equation}
(ij,kl) \coloneqq (kl,ij). \label{eq:ijkl_notation}
\end{equation}
This notation is also justified since we work with self-adjoint linear transformations here.

For each $i = 1,\dots,5$, we let $s_i \coloneqq c_i + c_{i\plmod 1} + c_{i \plmod 2}$ and $I_i \coloneqq \{(i,i),(i,i\plmod 1),(i\plmod 1,i\plmod 1),(i,i\plmod 2),(i\plmod 1,i\plmod 2),(i\plmod 2,i\plmod 2)\}$, where we use the notation introduced in \eqref{eq:ij_notation}.
It follows from \eqref{eq:Peirce_truncated} that
\begin{equation}
\bbE(s_i,1) = \bigoplus_{(k,l) \in I_i}\bbE_{kl}. \label{eq:Esi1}
\end{equation}
In addition, we define $H_i \coloneqq P_{\bbE(s_i,1)}H|_{\bbE(s_i,1)}$.
The mapping $H_i$ is the principal subtransformation of $H$ obtained by extracting the rows and columns indexed by the elements in $I_i$,
and $H_i = (c_i - c_{i\plmod 1} + c_{i\plmod 2})^{\otimes 2}$ holds.
Since
\begin{equation*}
c_i - c_{i\plmod 1} + c_{i\plmod 2} \in \bbE(s_i,1) \setminus (\bbE(s_i,1)_+ \cup (-\bbE(s_i,1)_+))
\end{equation*}
and the cone $\bbE(s_i,1)_+$ is symmetric in the space $\bbE(s_i,1)$, it follows from \cite[Proposition~5.2]{NL2026} that $\bbR_+H_i$ is an exposed ray of $\COP(\bbE(s_i,1)_+;\bbE(s_i,1))$.
This means that there exists $D_i \in \CP(\bbE(s_i,1)_+)$ such that
\begin{equation}
\bbR_+H_i = \COP(\bbE(s_i,1)_+) \cap \{D_i\}^\perp. \label{eq:Hi_exposed}
\end{equation}
We define $D \coloneqq \sum_{i=1}^5 D_i \in \CP(\bbE_+)$.
In what follows, we show that $\bbR_+H = \COP(\bbE_+) \cap \{D\}^\perp$ to prove that $\bbR_+ H$ is an exposed ray of $\COP(\bbE_+)$.

We have shown that $H \in \COP(\bbE_+)$ in Lemma~\ref{lem:Horn_cop}.
In addition, $H \in \{D\}^\perp$ holds since
\begin{equation*}
\langle H,D\rangle = \sum_{i=1}^5 \langle H,D_i\rangle = \sum_{i=1}^5\langle H_i,D_i\rangle = 0,
\end{equation*}
where we use Lemma~\ref{lem:A_dot_B_reduced} to derive the second equality and use \eqref{eq:Hi_exposed} to derive the last equality.
Therefore, we have $H \in \COP(\bbE_+) \cap \{D\}^\perp$.

Conversely, we let $A \in \COP(\bbE_+) \cap \{D\}^\perp$.
For each $i = 1,\dots,5$, we define $A_i \coloneqq P_{\bbE(s_i,1)}A|_{\bbE(s_i,1)}$.
Then each $A_i$ belongs to $\COP(\bbE(s_i,1)_+)$ by Lemma~\ref{lem:subtf_cop} and it follows from Lemma~\ref{lem:A_dot_B_reduced} that
\begin{equation*}
0 = \langle A, D\rangle = \sum_{i=1}^5\langle A_i, D_i\rangle.
\end{equation*}
For every $i = 1,\dots,5$, the value $\langle A_i, D_i\rangle$ is nonnegative, so it must be $0$.
Combining this with \eqref{eq:Hi_exposed}, we see that there exists $a_i \in \bbR_+$ such that $A_i = a_i H_i$.
For each $i = 1,\dots,5$ and $j = 0,1,2$, it follows from $c_{i\plmod j} \in \bbE(s_i,1)$ that
\begin{equation*}
c_{i\plmod j} \bullet A(c_{i\plmod j}) = c_{i\plmod j} \bullet A_i(c_{i\plmod j}) = c_{i\plmod j} \bullet (a_i H_i)(c_{i\plmod j}) = a_i \norm{c_{i\plmod j}}^4.
\end{equation*}
In particular, for any $i = 1,\dots,5$ and $j = 0,1,2$ such that $i\plmod j$ takes the same value, $c_{i\plmod j} \bullet A(c_{i\plmod j})$ also takes the same value, so the following equations hold for any $i = 1,\dots,5$:
\begin{equation*}
c_i\bullet A(c_i) = a_i \norm{c_i}^4 = a_{i\plmod 4} \norm{c_i}^4 = a_{i\plmod 3} \norm{c_i}^4.
\end{equation*}
This implies that $a_1 = a_2 = a_3 = a_4 = a_5$, which we write as $a \in \bbR_+$.
Therefore, it follows from \eqref{enum:element_equal} of Lemma~\ref{lem:submatrix_zero} and \eqref{eq:Esi1} that
\begin{equation}
A_{ij,kl} = aH_{ij,kl} \text{ for all $(ij,kl) \in \calJ \coloneqq \bigcup_{p=1}^5 \{(ij,kl) \in I_p\times I_p \mid ij \preceq kl\}$}. \label{eq:A_eq_aH_on_size3block}
\end{equation}

For every $(ij,kl) \in \llbracket 5\rrbracket \setminus \calJ$, the $(ij,kl)$th element of $H$ is $0$.
Therefore, we would like to see that $A_{ij,kl} = 0$ for all $(ij,kl) \in \llbracket 5\rrbracket \setminus \calJ$ to show that $A = aH$.
The indices $(ij,kl)$ in $\llbracket 5\rrbracket \setminus \calJ$ can be divided into the following seven cases:
\begin{enumerate}[C{a}se~a:]
\item $(ij,kl) = (11,24)$, $(22,35)$, $(33,14)$, $(44,25)$, $(13,55)$, $(11,35)$, $(22,14)$, $(33,25)$, $(13,44)$, and $(24,55)$;\label{enum:case_a}
\item $(ij,kl) = (11,34)$, $(22,45)$, $(33,15)$, $(12,44)$, and $(23,55)$;\label{enum:case_b}
\item $(ij,kl) = (12,24)$, $(23,35)$, $(14,34)$, $(25,45)$, $(13,15)$, $(15,35)$, $(12,14)$, $(23,25)$, $(13,34)$, and $(24,45)$;\label{enum:case_c}
\item $(ij,kl) = (13,14)$, $(24,25)$, $(13,35)$, $(14,24)$, and $(25,35)$;\label{enum:case_d}
\item $(ij,kl) = (12,34)$, $(23,45)$, $(34,15)$, $(12,45)$, and $(23,15)$;\label{enum:case_e}
\item $(ij,kl) = (12,35)$, $(23,14)$, $(34,25)$, $(13,45)$, and $(24,15)$;\label{enum:case_f}
\item $(ij,kl) = (13,24)$, $(24,35)$, $(14,35)$, $(14,25)$, and $(13,25)$. \label{enum:case_g}
\end{enumerate}
\begin{figure}
\begin{center}
\includegraphics{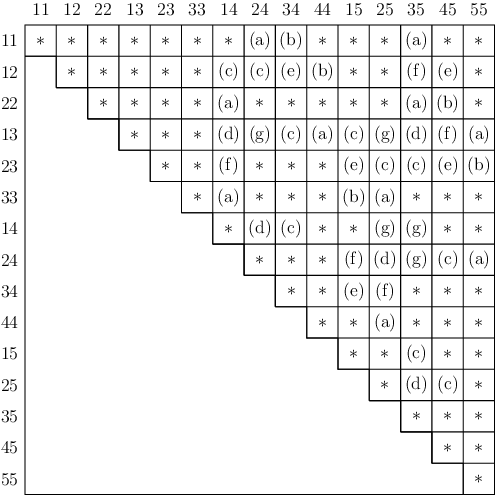}
\end{center}
\caption{Illustration of the seven cases from Case~\ref{enum:case_a} to Case~\ref{enum:case_g}.
The cells marked with $*$ correspond to the indices $(ij,kl)$ in $\calJ$, for which $A_{ij,kl} = aH_{ij,kl}$ holds as shown in \eqref{eq:A_eq_aH_on_size3block}.}
\label{fig:Horn_exposed_rank_5}
\end{figure}
See also Figure~\ref{fig:Horn_exposed_rank_5} for the above case distinction.
For simplicity, we define $t_i \coloneqq 1/\norm{c_i}$ for each $i = 1,\dots,5$.
Analogously to $\plmod$, define $\mnmod$ to be the subtraction modulo $5$ on $\{1,\dots,5\}$ (e.g., $1\mnmod 2 = 4$).

We recall the formulae in \eqref{eq:x^2} and \eqref{eq:ci_ci_xij2}, which are used in the following calculations.
In addition, we adopt the following notation.
In the matrix representation of a linear transformation,
the element marked with the symbol $(\bigstar)$ is $0$ because of Case~$\bigstar$.
For example, in \eqref{eq:case3_qf_matrix}, the $(11,24)$th and $(12,44)$th elements are $0$ because of Case~\ref{enum:case_a} and Case~\ref{enum:case_b}, respectively.

\noindent
\fbox{Case~\ref{enum:case_a}}
Under the notation introduced in \eqref{eq:ij_notation} and \eqref{eq:ijkl_notation}, the indices in $\llbracket 5\rrbracket \setminus \calJ$ that fall into this case can be summarized as $((i,i),(i\plmod 1,i\plmod 3))$ and $((i,i),(i\mnmod 1,i\mnmod 3))$ for $i = 1,\dots,5$.
By symmetry, showing $A_{11,24} = 0$ is sufficient.
For any $x_{24} \in \bbE_{24}$ and $\epsilon > 0$, let
\begin{align*}
x(\epsilon) &\coloneqq t_1^2c_1 + (t_2c_2 \pm \epsilon x_{24})^2 \\
&= \underbrace{t_1^2c_1}_{\in \bbE_{11}} + \underbrace{t_2^2c_2 + \epsilon^2c_2\circ x_{24}^2}_{\in \bbE_{22}} \pm \underbrace{\epsilon t_2x_{24}}_{\in \bbE_{24}} + \underbrace{\epsilon^2 c_4 \circ x_{24}^2}_{\in \bbE_{44}} \in \bbE_+.
\end{align*}
Then it follows from $A\in \COP(\bbE_+)$ that
\begin{align}
0 &\le q_A(x(\epsilon)) \nonumber\\
&= x(\epsilon)\bullet\!\!\!\bordermatrix{
& 11 & 22 & 24 & 44 \cr
& ac_1^{\otimes 2} & -ac_1\otimes c_2 & A_{11,24} & ac_1\otimes c_4 \cr
&  & ac_2^{\otimes 2} & 0 & ac_2\otimes c_4 \cr
&  &  & 0 & 0 \cr
&  &  &   & ac_4^{\otimes 2}}(x(\epsilon)) \nonumber\\
&= \pm 2\epsilon t_1^2t_2 c_1\bullet A_{11,24}(x_{24}) + O(\epsilon^2). \label{eq:case1_qf}
\end{align}
Dividing \eqref{eq:case1_qf} by $2\epsilon t_1^2t_2$ and letting $\epsilon \downarrow 0$, we have $c_1\bullet A_{11,24}(x_{24}) = 0$.
Since $\bbE_{11} = \bbR c_1$ and $x_{24} \in \bbE_{24}$ is arbitrary, by \eqref{enum:element_0} of Lemma~\ref{lem:submatrix_zero}, we obtain $A_{11,24} = 0$.

\noindent
\fbox{Case~\ref{enum:case_b}}
The indices in $\llbracket 5\rrbracket \setminus \calJ$ that fall into this case can be summarized as $((i,i),(i\plmod 2, i\plmod 3))$ for $i = 1,\dots,5$.
By symmetry, showing $A_{11,34} = 0$ is sufficient.
For any $x_{34} \in \bbE_{34}$ and $\epsilon > 0$, let
\begin{align*}
x(\epsilon) &\coloneqq t_1^2c_1 + 2t_2^2c_2 + (t_3 c_3 \pm \epsilon x_{34})^2\\
&= \underbrace{t_1^2c_1}_{\in \bbE_{11}} + \underbrace{2t_2^2c_2}_{\in \bbE_{22}} + \underbrace{t_3^2c_3 + \epsilon^2 c_3 \circ x_{34}^2}_{\in \bbE_{33}} \pm \underbrace{\epsilon t_3 x_{34}}_{\in \bbE_{34}} + \underbrace{\epsilon^2 c_4 \circ x_{34}^2}_{\in \bbE_{44}} \in \bbE_+.
\end{align*}
Then it follows that
\begin{align}
0 &\le  q_A(x(\epsilon)) \nonumber\\
&= x(\epsilon)\bullet\!\!\!\bordermatrix{
& 11 & 22 & 33 & 34 & 44 \cr
& ac_1^{\otimes 2} & -ac_1\otimes c_2 & ac_1\otimes c_3 &A_{11,34} & ac_1\otimes c_4 \cr
&  & ac_2^{\otimes 2} & -ac_2\otimes c_3 & 0 & ac_2\otimes c_4 \cr
&  &  & ac_3^{\otimes 2} & 0 & -ac_3\otimes c_4 \cr
&  &  &  & 0 & 0 \cr
&  &  &  & & ac_4^{\otimes 2}}(x(\epsilon)) \nonumber\\
&= \pm 2\epsilon t_1^2t_3 c_1\bullet A_{11,34}(x_{34}) + O(\epsilon^2). \label{eq:case2_qf}
\end{align}
Dividing \eqref{eq:case2_qf} by $2\epsilon t_1^2t_3$ and letting $\epsilon \downarrow 0$, we have $c_1\bullet A_{11,34}(x_{34}) = 0$.
Since $x_{34} \in \bbE_{34}$ is arbitrary, we obtain $A_{11,34} = 0$.

\noindent
\fbox{Case~\ref{enum:case_c}}
The indices in $\llbracket 5\rrbracket \setminus \calJ$ that fall into this case can be summarized as $((i,i\plmod 1),(i\plmod 1,i\plmod 3))$ and $((i,i\mnmod 1),(i\mnmod 1,i\mnmod 3))$ for $i = 1,\dots,5$.
By symmetry, showing $A_{12,24} = 0$ is sufficient.
For any $x_{12}\in \bbE_{12}$, $x_{24} \in \bbE_{24}$, and $\epsilon >  0$, let
\begin{align*}
x(\epsilon) &\coloneqq 2t_1^2c_1 + (t_1c_1 + x_{12} + t_2c_2)^2 + 2(t_2c_2 \pm \epsilon x_{24})^2\\
&= \underbrace{3t_1^2c_1 + c_1\circ x_{12}^2}_{\in \bbE_{11}} + \underbrace{(t_1 + t_2)x_{12}}_{\in \bbE_{12}} \\
&\quad + \underbrace{3t_2^2c_2 + c_2\circ x_{12}^2 + 2\epsilon^2 c_2 \circ x_{24}^2}_{\in \bbE_{22}} \pm \underbrace{2\epsilon t_2x_{24}}_{\in \bbE_{24}} + \underbrace{2\epsilon^2 c_4 \circ x_{24}^2}_{\in \bbE_{44}} \in \bbE_+.
\end{align*}
Then it follows that
\begin{align}
0 &\le q_A(x(\epsilon)) \nonumber\\
&= x(\epsilon)\bullet\!\!\!\bordermatrix{
& 11 & 12 & 22 & 24 & 44 \cr
& ac_1^{\otimes 2} & 0 & -ac_1\otimes c_2  & \eqref{enum:case_a} & ac_1\otimes c_4 \cr
&  & 0 & 0 & A_{12,24} & \eqref{enum:case_b} \cr
&  &  & ac_2^{\otimes 2} & 0 & ac_2\otimes c_4 \cr
&  &  &  & 0 & 0 \cr
&  &  &  & & ac_4^{\otimes 2}}(x(\epsilon)) \label{eq:case3_qf_matrix}\\
&= \pm 4\epsilon (t_1 + t_2)t_2 x_{12}\bullet A_{12,24}(x_{24}) + O(\epsilon^2). \label{eq:case3_qf}
\end{align}
Dividing \eqref{eq:case3_qf} by $4\epsilon (t_1 + t_2)t_2$ and letting $\epsilon \downarrow 0$, we have $x_{12}\bullet A_{12,24}(x_{24}) = 0$.
Since $x_{12}\in \bbE_{12}$ and $x_{24} \in \bbE_{24}$ are arbitrary, we obtain $A_{12,24} = 0$.

\noindent
\fbox{Case~\ref{enum:case_d}}
The indices in $\llbracket 5\rrbracket \setminus \calJ$ that fall into this case can be summarized as $((i,i\plmod 2),(i,i\plmod 3))$ for $i = 1,\dots,5$.
By symmetry, showing $A_{13,14} = 0$ is sufficient.
For any $x_{13}\in \bbE_{13}$, $x_{14} \in \bbE_{14}$, and $\epsilon >  0$, let
\begin{align*}
x(\epsilon) &\coloneqq (6 + \norm{x_{13}}^2)t_2^2c_2 + 2t_3^2c_3 + (t_1c_1 + x_{13} + t_3c_3)^2 + 2(t_1c_1 \pm \epsilon x_{14})^2 \\
&= \underbrace{3t_1^2c_1 + c_1\circ x_{13}^2 + 2\epsilon^2 c_1\circ x_{14}^2}_{\in \bbE_{11}} + \underbrace{(6 + \norm{x_{13}}^2)t_2^2c_2}_{\in \bbE_{22}} \\
&\quad + \underbrace{(t_1 + t_3)x_{13}}_{\in \bbE_{13}} + \underbrace{3t_3^2c_3 + c_3\circ x_{13}^2}_{\in \bbE_{33}} \pm \underbrace{2\epsilon t_1x_{14}}_{\in \bbE_{14}} + \underbrace{2\epsilon^2 c_4\circ x_{14}^2}_{\in \bbE_{44}} \in \bbE_+.
\end{align*}
Then it follows that
\begin{align}
0 &\le q_A(x(\epsilon)) \nonumber\\
&= x(\epsilon)\bullet\!\!\!\bordermatrix{
& 11 & 22 & 13 & 33 & 14 & 44 \cr
& ac_1^{\otimes 2} & -ac_1\otimes c_2 & 0 & ac_1\otimes c_3 & 0 & ac_1\otimes c_4 \cr
&  & ac_2^{\otimes 2} & 0 & -ac_2\otimes c_3 & \eqref{enum:case_a} & ac_2\otimes c_4 \cr
&  &  & 0 & 0 & A_{13,14} & \eqref{enum:case_a} \cr
&  &  &  & ac_3^{\otimes 2} & \eqref{enum:case_a} & -ac_3\otimes c_4 \cr
&  &  &  & & 0 & 0 \cr
&  &  &  & & & ac_4^{\otimes 2}}(x(\epsilon)) \nonumber\\
&= \pm 4\epsilon (t_1 + t_3)t_1 x_{13}\bullet A_{13,14}(x_{14}) + O(\epsilon^2). \label{eq:case4_qf}
\end{align}
Dividing \eqref{eq:case4_qf} by $4\epsilon (t_1 + t_3)t_1$ and letting $\epsilon \downarrow 0$, we have $x_{13}\bullet A_{13,14}(x_{14}) = 0$.
Since $x_{13}\in \bbE_{13}$ and $x_{14} \in \bbE_{14}$ are arbitrary, we obtain $A_{13,14} = 0$.

\noindent
\fbox{Case~\ref{enum:case_e}}
The indices in $\llbracket 5\rrbracket \setminus \calJ$ that fall into this case can be summarized as $((i,i\plmod 1),(i\plmod 2, i\plmod 3))$ for $i = 1,\dots,5$.
By symmetry, showing $A_{12,34} = 0$ is sufficient.
For any $x_{12}\in \bbE_{12}$, $x_{34} \in \bbE_{34}$, and $\epsilon >  0$, let
\begin{align*}
x(\epsilon) &\coloneqq (t_2c_2 \pm \epsilon x_{12})^2 + (t_3c_3 + \epsilon x_{34})^2 \\
&= \underbrace{\epsilon^2 c_1 \circ x_{12}^2}_{\in \bbE_{11}} \pm \underbrace{\epsilon t_2x_{12}}_{\in \bbE_{12}} + \underbrace{t_2^2c_2 + \epsilon^2c_2\circ x_{12}^2}_{\in \bbE_{22}}\\
&\quad + \underbrace{t_3^2c_3 + \epsilon^2c_3 \circ x_{34}^2}_{\in \bbE_{33}} + \underbrace{\epsilon t_3x_{34}}_{\in \bbE_{34}} + \underbrace{\epsilon^2 c_4\circ x_{34}^2}_{\in \bbE_{44}} \in \bbE_+.
\end{align*}
Then it follows that
\begin{align}
0 &\le q_A(x(\epsilon)) \nonumber\\
&= x(\epsilon)\bullet\!\!\!\bordermatrix{
& 11 & 12 & 22 & 33 & 34 & 44 \cr
& ac_1^{\otimes 2} & 0 & -ac_1\otimes c_2 & ac_1\otimes c_3 & \eqref{enum:case_b} & ac_1\otimes c_4 \cr
&  & 0 & 0 & 0 & A_{12,34} & \eqref{enum:case_b} \cr
&  &  & ac_2^{\otimes 2} & -ac_2\otimes c_3 & 0 & ac_2\otimes c_4 \cr
&  &  &  & ac_3^{\otimes 2} & 0 & -ac_3\otimes c_4 \cr
&  &  &  & & 0 & 0 \cr
&  &  &  & & & ac_4^{\otimes 2}}(x(\epsilon)) \nonumber\\
&= \pm 2\epsilon^2 t_2t_3 x_{12}\bullet A_{12,34}(x_{34}) + O(\epsilon^4). \label{eq:case5_qf}
\end{align}
Dividing \eqref{eq:case5_qf} by $2\epsilon^2 t_2t_3$ and letting $\epsilon \downarrow 0$, we have $x_{12}\bullet A_{12,34}(x_{34}) = 0$.
Since $x_{12}\in \bbE_{12}$ and $x_{34} \in \bbE_{34}$ are arbitrary, we obtain $A_{12,34} = 0$.

\noindent
\fbox{Case~\ref{enum:case_f}}
The indices in $\llbracket 5\rrbracket \setminus \calJ$ that fall into this case can be summarized as $((i,i\plmod 1),(i\plmod 2,i\plmod 4))$ for $i = 1,\dots,5$.
By symmetry, showing $A_{12,35} = 0$ is sufficient.
For any $x_{12}\in \bbE_{12}$, $x_{35} \in \bbE_{35}$, and $\epsilon >  0$, let
\begin{align*}
x(\epsilon) &\coloneqq (t_1c_1 \pm x_{12} + t_2c_2)^2 + \epsilon(t_3c_3 + x_{35} + t_5c_5)^2\\
&= \underbrace{t_1^2c_1 + c_1\circ x_{12}^2}_{\in \bbE_{11}} \pm \underbrace{(t_1 + t_2)x_{12}}_{\in \bbE_{12}} + \underbrace{t_2^2c_2 + c_2\circ x_{12}^2}_{\in \bbE_{22}} \\
&\quad + \underbrace{\epsilon t_3^2c_3 + \epsilon c_3\circ x_{35}^2}_{\in \bbE_{33}} + \underbrace{\epsilon(t_3 + t_5)x_{35}}_{\in \bbE_{35}} + \underbrace{\epsilon t_5^2c_5 + \epsilon c_5\circ x_{35}^2}_{\in \bbE_{55}} \in \bbE_+.
\end{align*}
Then it follows that
\begin{align}
0 &\le q_A(x(\epsilon)) \nonumber\\
&= x(\epsilon)\bullet\!\!\!\bordermatrix{
& 11 & 12 & 22 & 33 & 35 & 55 \cr
& ac_1^{\otimes 2} & 0 & -ac_1\otimes c_2 & ac_1\otimes c_3 & \eqref{enum:case_a} & -ac_1\otimes c_5 \cr
&  & 0 & 0 & 0 & A_{12,35} & 0 \cr
&  &  & ac_2^{\otimes 2} & -ac_2\otimes c_3 & \eqref{enum:case_a} & ac_2\otimes c_5 \cr
&  &  &  & ac_3^{\otimes 2} & 0 & ac_3\otimes c_5 \cr
&  &  &  & & 0 & 0 \cr
&  &  &  & & & ac_5^{\otimes 2}}(x(\epsilon)) \nonumber\\
&= \pm 2\epsilon(t_1+t_2)(t_3+t_5) x_{12}\bullet A_{12,35}(x_{35}) + O(\epsilon^2). \label{eq:case6_qf}
\end{align}
Dividing \eqref{eq:case6_qf} by $2\epsilon(t_1+t_2)(t_3+t_5)$ and letting $\epsilon \downarrow 0$, we have $x_{12}\bullet A_{12,35}(x_{35}) = 0$.
Since $x_{12}\in \bbE_{12}$ and $x_{35} \in \bbE_{35}$ are arbitrary, we obtain $A_{12,35} = 0$.

\noindent
\fbox{Case~\ref{enum:case_g}}
The indices in $\llbracket 5\rrbracket \setminus \calJ$ that fall into this case can be summarized as $((i,i\plmod 2),(i\plmod 1,i\plmod 3))$ for $i = 1,\dots,5$.
By symmetry, showing $A_{13,24} = 0$ is sufficient.
For any $x_{13}\in \bbE_{13}$, $x_{24} \in \bbE_{24}$, and $\epsilon >  0$, let
\begin{align*}
x(\epsilon) &\coloneqq (t_2c_2 + \epsilon x_{24})^2 + (t_3c_3 \pm \epsilon x_{13})^2\\
&= \underbrace{\epsilon^2 c_1\circ x_{13}^2}_{\in \bbE_{11}} + \underbrace{t_2^2c_2 + \epsilon^2 c_2\circ x_{24}^2}_{\in \bbE_{22}} \pm \underbrace{\epsilon t_3 x_{13}}_{\in \bbE_{13}}\\
&\quad + \underbrace{t_3^2c_3 + \epsilon^2c_3\circ x_{13}^2}_{\in \bbE_{33}} + \underbrace{\epsilon t_2x_{24}}_{\in \bbE_{24}} + \underbrace{\epsilon^2 c_4\circ x_{24}^2}_{\in \bbE_{44}} \in \bbE_+.
\end{align*}
Then it follows that
\begin{align}
0 &\le q_A(x(\epsilon)) \nonumber\\
&= x(\epsilon)\bullet\!\!\!\bordermatrix{
& 11 & 22 & 13 & 33 & 24 & 44 \cr
& ac_1^{\otimes 2} & -ac_1\otimes c_2 & 0 & ac_1\otimes c_3 & \eqref{enum:case_a} & ac_1\otimes c_4 \cr
&  & ac_2^{\otimes 2} & 0 & -ac_2\otimes c_3 & 0 & ac_2\otimes c_4 \cr
&  &  & 0 & 0 & A_{13,24} & \eqref{enum:case_a} \cr
&  &  &  & ac_3^{\otimes 2} & 0 & -ac_3\otimes c_4 \cr
&  &  &  & & 0 & 0 \cr
&  &  &  & & & ac_4^{\otimes 2}}(x(\epsilon)) \nonumber\\
&= \pm 2\epsilon^2 t_2t_3 x_{13}\bullet A_{13,24}(x_{24}) + O(\epsilon^4). \label{eq:case7_qf}
\end{align}
Dividing \eqref{eq:case7_qf} by $2\epsilon^2 t_2t_3$ and letting $\epsilon \downarrow 0$, we have $x_{13}\bullet A_{13,24}(x_{24}) = 0$.
Since $x_{13}\in \bbE_{13}$ and $x_{24} \in \bbE_{24}$ are arbitrary, we obtain $A_{13,24} = 0$.
\end{proof}

\begin{remark}
In Lemma~\ref{lem:Horn_exposed_rank5}, when the Euclidean Jordan algebra $\bbE$ is the Hadamard Euclidean Jordan algebra, no case distinction is required in the proof.
In this situation, $\bbE_{ij}$ is the set containing only zero for all $1 \le i < j \le 5$, and the proof of Lemma~\ref{lem:Horn_exposed_rank5} terminates at \eqref{eq:A_eq_aH_on_size3block}.
Moreover, the Jordan frame $C$ of the Hadamard Euclidean Jordan algebra is $\{\bm{e}_1,\dots,\bm{e}_5\}$ as noted in Section~\ref{subsec:EJA}, and the corresponding Horn transformation $H_C$ is essentially identical to the Horn matrix $\bm{H}$.
For this reason, proving Lemma~\ref{lem:Horn_exposed_rank5} is fundamentally more involved than proving that the Horn matrix generates an exposed ray of $\COP^5$.
\end{remark}

\begin{lemma}\label{lem:Horn_exposed}
Suppose that the rank of the Euclidean Jordan algebra $\bbE$ is at least $5$.
For a set $C = \{c_1,\dots,c_5\}$ of orthogonal primitive idempotents in $\bbE$, the Horn transformation $H_C$ generates an exposed ray of $\COP(\bbE_+)$.
\end{lemma}

\begin{proof}
We prove the statement by induction on the rank of $\bbE$.
When the rank of $\bbE$ is $5$, the statement holds by Lemma~\ref{lem:Horn_exposed_rank5}.
We assume that the statement holds for Euclidean Jordan algebras of rank $r \ge 5$.
For a Euclidean Jordan algebra $\bbE$ of rank $r_+\coloneqq r + 1$, we complement the set $C$ so that $\{c_1,\dots,c_{r_+}\}$ is a Jordan frame of $\bbE$.
We consider the Peirce decomposition~\eqref{eq:Peirce} of $\bbE$ with respect to this Jordan frame.
By the orthogonality of the Peirce decomposition, $H_C$, denoted by $H$ hereafter, satisfies
\begin{equation}
H(x) = 0 \text{ for all $(i,j) \not\in \{(1,1),(2,2),(3,3),(4,4),(5,5)\}$ and $x\in \bbE_{ij}$}. \label{eq:Hx=0}
\end{equation}
Let $c \coloneqq \sum_{i=1}^r c_i$.
Then $\{c_1,\dots,c_r\}$ is a Jordan frame of the Euclidean Jordan subalgebra $\bbE(c,1)$ of rank $r$.
We note that $C$ is also a set of orthogonal primitive idempotents in $\bbE(c,1)$, and we can view $H$ as an element in $\calS(\bbE(c,1))$.
By the inductive assumption, there exists $D_1 \in \CP(\bbE(c,1)_+)$ such that
\begin{equation}
\bbR_+ H = \COP(\bbE(c,1)_+;\bbE(c,1)) \cap \{D_1\}^\perp. \label{eq:Horn_exposed_rankr}
\end{equation}

When we write $D_1$ as $D_1 = \sum_{i=1}^m d_i^{\otimes 2}$ for some $d_1,\dots,d_m \in \bbE(c,1)_+$, \eqref{eq:Horn_exposed_rankr} implies that
\begin{equation}
d_i\bullet H(d_i) = 0 \label{eq:Horn_qf_at_di}
\end{equation}
for every $i = 1,\dots,m$.
In addition, the elements $d_1,\dots,d_m$ satisfy
\begin{equation}
\setspan\{d_1,\dots,d_m\} = \bbE(c,1). \label{eq:span_di}
\end{equation}
Indeed, if $\setspan\{d_1,\dots,d_m\}$ is strictly included in $\bbE(c,1)$, there exists $d \in \bbE(c,1) \setminus \{0\}$ such that $d\bullet d_i = 0$ for every $i = 1,\dots,m$.
Then it follows that $\langle d^{\otimes 2}, D_1\rangle = 0$.
In addition, we see that $d^{\otimes 2} \in \COP(\bbE(c,1)_+)$.
Therefore, it follows from \eqref{eq:Horn_exposed_rankr} that $d^{\otimes 2} \in \bbR_+ H$.
Since $H$ is not positive semidefinite as shown in Lemma~\ref{lem:Horn_not_psd}, $d^{\otimes 2}$ must be $0$, which contradicts $d\neq 0$.

Let $D_2 \coloneqq c_{r_+}^{\otimes 2} \in \CP(\bbE_+)$ and $D_3 \coloneqq \sum_{i=1}^m (d_i + c_{r_+})^{\otimes 2} \in \CP(\bbE_+)$.
For each $i = 1,\dots,5$, we let $t_i \coloneqq 1/\norm{c_i}$, let $\{u_{i1},\dots,u_{in_i}\}$ be a basis for the space $\bbE_{ir_+}$, and let
\begin{align*}
D_4 &\coloneqq \sum_{i=1}^5 \sum_{1\le j\le k\le n_i}[
2\{t_ic_i + (u_{ij} + u_{ik})\}^2 + (2 + \norm{u_{ij} + u_{ik}}^2)t_{i_{\plmod}}^2c_{i_{\plmod}}]^{\otimes 2}\\
&\quad + \sum_{i=1}^5 \sum_{1\le j\le k\le n_i}[
2\{t_ic_i - (u_{ij} + u_{ik})\}^2 + (2 + \norm{u_{ij} + u_{ik}}^2)t_{i_{\plmod}}^2c_{i_{\plmod}}]^{\otimes 2}\\
&\in \CP(\bbE_+),
\end{align*}
where we write $i_{\plmod}$ for $i\plmod 1$ for notational convenience.
In addition, for each $i = 6,\dots,r$, let $\{u_{i1},\dots,u_{in_i}\}$ be a basis for the space $\bbE_{ir_+}$ and let
\begin{align*}
D_5 &\coloneqq \sum_{i=6}^r\sum_{1\le j\le k\le n_i} [\{c_i + (u_{ij} + u_{ik})\}^2]^{\otimes 2} \\
&\qquad + \sum_{i=6}^r\sum_{1\le j\le k\le n_i} [\{c_i - (u_{ij} + u_{ik})\}^2]^{\otimes 2} \in \CP(\bbE_+).
\end{align*}
Finally, we define $D \coloneqq \sum_{i=1}^5 D_i \in \CP(\bbE_+)$.
In what follows, we show that $\bbR_+H = \COP(\bbE_+) \cap \{D\}^\perp$ to prove that $\bbR_+ H$ is an exposed ray of $\COP(\bbE_+)$.

It follows from Lemma~\ref{lem:Horn_cop} that $H \in \COP(\bbE_+)$.
We show that $\langle H,D_i\rangle = 0$ holds for every $i = 1,\dots,5$ to prove $H \in \{D\}^\perp$.
Firstly, it follows from \eqref{eq:Horn_exposed_rankr} that $\langle H,D_1\rangle = 0$.
Secondly, by $r_+ \ge 6$ and \eqref{eq:Hx=0}, we have $\langle H,D_2\rangle = 0$.
Thirdly, we have
\begin{equation*}
\langle H,D_3\rangle = \sum_{i=1}^m d_i\bullet H(d_i) + \sum_{i=1}^m 2d_i\bullet H(c_{r_+}) + mc_{r_+}\bullet H(c_{r_+}) = 0,
\end{equation*}
where in the second equality, the first term vanishes by \eqref{eq:Horn_qf_at_di}, and the second and third terms vanish by $r_+ \ge 6$ and \eqref{eq:Hx=0}.
Fourthly, for each $i = 1,\dots,5$ and $1\le j\le k\le n_i$, we see that
\begin{align}
&2\{t_ic_i \pm (u_{ij} + u_{ik})\}^2 + (2 + \norm{u_{ij} + u_{ik}}^2)t_{i_{\plmod}}^2c_{i_{\plmod}} \nonumber\\
&\quad = \underbrace{2t_i^2c_i + 2c_i\circ (u_{ij} + u_{ik})^2}_{\in \bbE_{ii}} \nonumber\\
&\qquad + \underbrace{(2 + \norm{u_{ij} + u_{ik}}^2)t_{i_{\plmod}}^2c_{i_{\plmod}}}_{\in \bbE_{i_{\plmod} i_{\plmod}}} \pm \underbrace{2t_i(u_{ij} + u_{ik})}_{\in \bbE_{ir_+}} + \underbrace{2c_{r_+}\circ (u_{ij} + u_{ik})^2}_{\in \bbE_{r_+r_+}}. \label{eq:D4_comp_vec}
\end{align}
The principal subtransformation of $H$ corresponding to the indices $ii$, $i_{\plmod}i_{\plmod}$, $ir_+$, and $r_+r_+$ is
\begin{equation*}
\bordermatrix{
& ii & i_{\plmod}i_{\plmod} & ir_+ & r_+r_+ \cr
& c_i^{\otimes 2} & -c_i\otimes c_{i_{\plmod}} & 0 & 0 \cr
&  & c_{i_{\plmod}}^{\otimes 2} & 0 & 0 \cr
&  &  & 0 & 0 \cr
&  &  &   & 0},
\end{equation*}
and
\begin{equation*}
\{2t_i^2c_i + 2c_i\circ (u_{ij} + u_{ik})^2\}\bullet c_i = \{(2 + \norm{u_{ij} + u_{ik}}^2)t_{i_{\plmod}}^2c_{i_{\plmod}}\} \bullet c_{i_{\plmod}} = 2  + \norm{u_{ij} + u_{ik}}^2,
\end{equation*}
where the equality between the first and third expressions follows from \eqref{eq:ci_ci_xij2}.
Therefore, we have $\langle H,D_4\rangle = 0$.
Fifthly, it follows from \eqref{eq:Hx=0} that $\langle H,D_5\rangle = 0$.

Conversely, we let $A \in \COP(\bbE_+) \cap \{D\}^\perp$.
Then it follows that $\langle A,D_i\rangle = 0$ for each $i = 1,\dots,5$.
Firstly, it follows from Lemma~\ref{lem:A_dot_B_reduced} that
\begin{equation*}
0 = \langle A,D_1\rangle = \langle P_{\bbE(c,1)}A|_{\bbE(c,1)},D_1\rangle = \sum_{i=1}^m d_i\bullet A(d_i).
\end{equation*}
In particular, we have
\begin{equation}
d_i\bullet A(d_i) = 0 \label{eq:A_qf_di=0}
\end{equation}
for every $i = 1,\dots,m$.
In addition, $P_{\bbE(c,1)}A|_{\bbE(c,1)} \in \COP(\bbE(c,1)_+)$ holds by Lemma~\ref{lem:subtf_cop}.
Therefore, by \eqref{eq:Horn_exposed_rankr}, there exists $a\in \bbR_+$ such that $P_{\bbE(c,1)}A|_{\bbE(c,1)} = aH$.
Since
\begin{equation}
\bbE(c,1) = \bigoplus_{1\le i\le j\le r}\bbE_{ij} \label{eq:Ec1_Peirce}
\end{equation}
holds by \eqref{eq:Peirce_truncated}, we see from \eqref{enum:element_equal} of Lemma~\ref{lem:submatrix_zero} that
\begin{equation}
A_{ij,kl} = aH_{ij,kl} \text{ for all $(ij,kl) \in \llbracket r\rrbracket$}. \label{eq:A_eq_aH_on_Ec1}
\end{equation}
Secondly, $\langle A,D_2\rangle = 0$ implies that
\begin{equation}
0 = \langle A, c_{r_+}^{\otimes 2}\rangle = c_{r_+}\bullet A(c_{r_+}) = c_{r_+}\bullet A_{r_+r_+,r_+r_+}(c_{r_+}). \label{eq:A_dot_D2}
\end{equation}
By $\bbE_{r_+r_+} = \bbR c_{r_+}$, it follows from \eqref{enum:element_0} of Lemma~\ref{lem:submatrix_zero} that
\begin{equation}
A_{r_+r_+,r_+r_+} = 0. \label{eq:Ar+r+r+r+_eq_0}
\end{equation}
Thirdly, $\langle A,D_3\rangle = 0$ implies that
\begin{align*}
0 &= \langle A,(d_i + c_{r_+})^{\otimes 2} \rangle\\
  &= d_i\bullet A(d_i) + 2d_i \bullet A(c_{r_+}) +  c_{r_+}\bullet A(c_{r_+})\\
&= 2d_i \bullet (P_{\bbE(c,1)}A|_{\bbE_{r_+r_+}})(c_{r_+}),
\end{align*}
for each $i = 1,\dots,m$, where the third equality follows from \eqref{eq:A_qf_di=0} and \eqref{eq:A_dot_D2}.
By \eqref{eq:span_di} and $\bbE_{r_+r_+} = \bbR c_{r_+}$, \eqref{enum:element_0} of Lemma~\ref{lem:submatrix_zero} implies that $P_{\bbE(c,1)}A|_{\bbE_{r_+r_+}} = 0$.
From \eqref{eq:Ec1_Peirce} and \eqref{enum:element_equal} of Lemma~\ref{lem:submatrix_zero}, we obtain
\begin{equation}
A_{ij,r_+r_+} = 0 \text{ for all $1\le i\le j\le r$}. \label{eq:Aijr+r+_eq_0}
\end{equation}
Fourthly, we note that for every $i = 1,\dots,5$, the principal subtransformation of $A$ corresponding to the indices $ii$, $i_{\plmod}i_{\plmod}$, $ir_+$, and $r_+r_+$ is
\begin{equation*}
\bordermatrix{
& ii & i_{\plmod}i_{\plmod} & ir_+ & r_+r_+ \cr
& ac_i^{\otimes 2} & -ac_i\otimes c_{i_{\plmod}} & A_{ii,ir_+} & \eqref{eq:Aijr+r+_eq_0} \cr
&  & ac_{i_{\plmod}}^{\otimes 2} & A_{i_{\plmod}i_{\plmod},ir_+} & \eqref{eq:Aijr+r+_eq_0} \cr
&  &  & A_{ir_+,ir_+} & A_{ir_+,r_+r_+} \cr
&  &  &   & \eqref{eq:Ar+r+r+r+_eq_0}},
\end{equation*}
where \eqref{eq:Ar+r+r+r+_eq_0} and \eqref{eq:Aijr+r+_eq_0} indicate that the corresponding elements are $0$ because of \eqref{eq:Ar+r+r+r+_eq_0} and \eqref{eq:Aijr+r+_eq_0}, respectively, and we also use this convention in the following proof.
From $\langle A,D_4\rangle = 0$, by using \eqref{eq:D4_comp_vec}, we have
\begin{align}
0 &= \langle A,[2\{t_ic_i \pm (u_{ij} + u_{ik})\}^2 + (2 + \norm{u_{ij} + u_{ik}}^2)t_{i_{\plmod}}^2c_{i_{\plmod}}]^{\otimes 2}\rangle \nonumber \\
&= \pm 8t_i(t_i^2c_i + c_i\circ (u_{ij} + u_{ik})^2)\bullet A_{ii,ir_+}(u_{ij} + u_{ik}) \nonumber \\
&\quad \pm 4t_it_{i_{\plmod}}^2(2 + \norm{u_{ij} + u_{ik}}^2)c_{i_{\plmod}}\bullet A_{i_{\plmod}i_{\plmod} ,ir_+}(u_{ij} + u_{ik}) \nonumber \\
&\quad + 4t_i^2\langle A_{ir_+,ir_+},(u_{ij} + u_{ik})^{\otimes 2}\rangle \nonumber\\
&\quad \pm 8t_i(u_{ij} + u_{ik}) \bullet A_{ir_+,r_+r_+}(c_{r_+}\circ (u_{ij} + u_{ik})^2) \label{eq:A_dot_D4_comp}
\end{align}
for every $i = 1,\dots,5$ and $1\le j\le k \le n_i$.
Adding \eqref{eq:A_dot_D4_comp} with the plus sign and that with the minus sign, and dividing the sum by $8t_i^2$, we have $\langle A_{ir_+,ir_+},(u_{ij} + u_{ik})^{\otimes 2}\rangle = 0$.
Since $\{u_{i1},\dots,u_{in_i}\}$ is a basis for $\bbE_{ir_+}$, the set $\{(u_{ij} + u_{ik})^{\otimes 2} \mid 1\le j\le k\le n_i\}$ is a basis for $\calS(\bbE_{ir_+})$~\cite[Lemma~6.2]{Dickinson2011}.
Therefore, it follows that
\begin{equation}
A_{ir_+,ir_+} = 0. \label{eq:Air+ir+_eq_0_ile5}
\end{equation}
Fifthly, from $\langle A,D_5 \rangle = 0$, we see that
\begin{align}
0 &= \langle A,[\{c_i \pm (u_{ij} + u_{ik})\}^2]^{\otimes 2}\rangle \nonumber \\
&= \{c_i \pm (u_{ij} + u_{ik})\}^2 \bullet\!\!\!\bordermatrix{
& ii & ir_+ & r_+r_+ \cr
& \eqref{eq:A_eq_aH_on_Ec1} & A_{ii,ir_+} & \eqref{eq:Aijr+r+_eq_0} \cr
&  & A_{ir_+,ir_+} & A_{ir_+,r_+r_+} \cr
&  &  & \eqref{eq:Ar+r+r+r+_eq_0}}(\{c_i \pm (u_{ij} + u_{ik})\}^2) \nonumber\\
&= \pm 2(c_i + c_i\circ (u_{ij} + u_{ik})^2)\bullet A_{ii,ir_+}(u_{ij} + u_{ik}) + \langle A_{ir_+,ir_+},(u_{ij} + u_{ik})^{\otimes 2}\rangle \nonumber \\
&\quad \pm 2(u_{ij} + u_{ik})\bullet A_{ir_+,r_+r_+}(c_{r_+} \circ (u_{ij} + u_{ik})^2) \label{eq:A_dot_D5_comp}
\end{align}
holds for every $i = 6,\dots,r$ and $1\le j\le k\le n_i$.
Adding \eqref{eq:A_dot_D5_comp} with the plus sign and that with the minus sign, and dividing the sum by $2$, we have $\langle A_{ir_+,ir_+},(u_{ij} + u_{ik})^{\otimes 2}\rangle = 0$.
Since $\{u_{i1},\dots,u_{in_i}\}$ is a basis for $\bbE_{ir_+}$, in the same way as in \eqref{eq:Air+ir+_eq_0_ile5}, it follows that
\begin{equation}
A_{ir_+,ir_+} = 0. \label{eq:Air+ir+_eq_0_ige6}
\end{equation}
From \eqref{eq:Air+ir+_eq_0_ile5} and \eqref{eq:Air+ir+_eq_0_ige6}, we obtain
\begin{equation}
A_{ir_+,ir_+} = 0 \text{ for all $1\le i\le r$}. \label{eq:Air+ir+_eq_0}
\end{equation}

To prove that $A = aH$, it suffices to show that the elements of $A$ corresponding to the following six types of indices $(ij,kl) \in \llbracket r_+\rrbracket$ are $0$:
\begin{enumerate}[C{a}se~a:]
\item $1\le i\le r$ and $j = k = l = r_+$;\label{enum:case_a_rkg}
\item $1\le i < k\le r$ and $j = l = r_+$;\label{enum:case_b_rkg}
\item $1\le i = j = k \le r$ and $l = r_+$;\label{enum:case_c_rkg}
\item $1\le i = j\le r$, $1\le k \le r$, $k \not\in \{i,j\}$, and $l = r_+$;\label{enum:case_d_rkg}
\item $1\le i < j \le r$, $k\in \{i,j\}$, and $l = r_+$;\label{enum:case_e_rkg}
\item $1\le i,j,k\le r$, $i,j,k$ are different from each other, and $l = r_+$.\label{enum:case_f_rkg}
\end{enumerate}
\begin{figure}
\begin{center}
\includegraphics{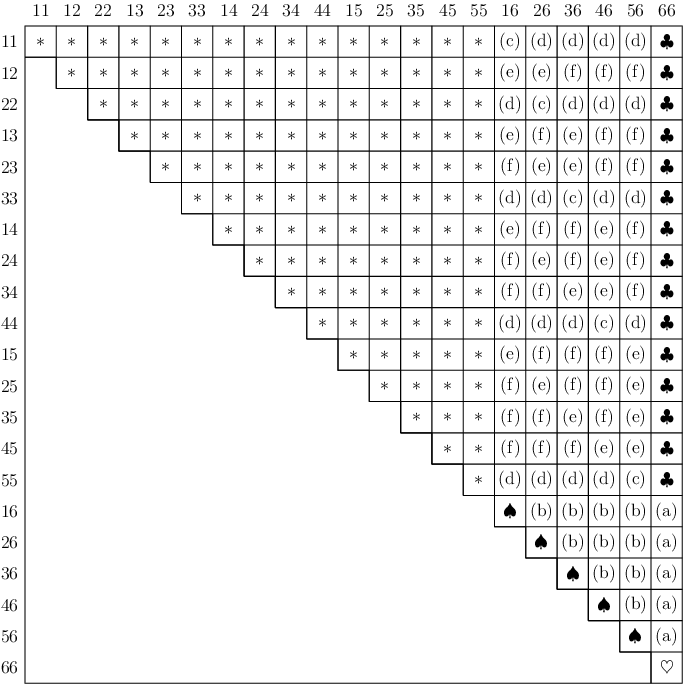}
\end{center}
\caption{Illustration of the six cases from Case~\ref{enum:case_a_rkg} to Case~\ref{enum:case_f_rkg} for $r = 5$.
The cells marked with $*$ correspond to the indices $(ij,kl) \in \llbracket r\rrbracket$, for which $A_{ij,kl} = aH_{ij,kl}$ holds as shown in \eqref{eq:A_eq_aH_on_Ec1}.
The cell marked with $\heartsuit$ indicates that the corresponding element is $0$ because of \eqref{eq:Ar+r+r+r+_eq_0}, those marked with $\clubsuit$ indicate that the corresponding elements are $0$ because of \eqref{eq:Aijr+r+_eq_0}, and those marked with $\spadesuit$ indicate that the corresponding elements are $0$ because of \eqref{eq:Air+ir+_eq_0}.}
\label{fig:Horn_exposed_rank_general}
\end{figure}
See Appendix for the reasoning of this case distinction.
See also Figure~\ref{fig:Horn_exposed_rank_general} for the above case distinction.

We recall the formula in \eqref{eq:x^2}, which is used in the following calculations.
In addition, as in Lemma~\ref{lem:Horn_exposed_rank5}, in the matrix representation of a linear transformation,
the element marked with the symbol $(\bigstar)$ is $0$ because of Case~$\bigstar$.
For example, in \eqref{eq:case_b_rkg_qf_mat}, the $(ir_+,r_+r_+)$th and $(kr_+,r_+r_+)$th elements are $0$ because of Case~\ref{enum:case_a_rkg}.

\noindent
\fbox{Case~\ref{enum:case_a_rkg}}
For any $x_{ir_+} \in \bbE_{ir_+}$ and $\epsilon > 0$, let
\begin{equation*}
x(\epsilon) \coloneqq (\pm \epsilon x_{ir_+} + c_{r_+})^2 = \underbrace{\epsilon^2 c_i\circ x_{ir_+}^2}_{\in \bbE_{ii}} \pm \underbrace{\epsilon x_{ir_+}}_{\in \bbE_{ir_+}} + \underbrace{c_{r_+} + \epsilon^2c_{r_+}\circ x_{ir_+}^2}_{\in \bbE_{r_+r_+}} \in \bbE_+.
\end{equation*}
Then it follows that
\begin{align}
0 &\le q_A(x(\epsilon)) \nonumber\\
&= x(\epsilon) \bullet\!\!\!\bordermatrix{
& ii & ir_+ & r_+r_+ \cr
& A_{ii,ii} & A_{ii,ir_+} & \eqref{eq:Aijr+r+_eq_0} \cr
&  & \eqref{eq:Air+ir+_eq_0} & A_{ir_+,r_+r_+} \cr
&  &  & \eqref{eq:Ar+r+r+r+_eq_0}
}(x(\epsilon)) \nonumber\\
&= \pm 2\epsilon x_{ir_+}\bullet A_{ir_+,r_+r_+}(c_{r_+}) + O(\epsilon^3). \label{eq:case1_rkg_qf}
\end{align}
Dividing \eqref{eq:case1_rkg_qf} by $2\epsilon$ and letting $\epsilon \downarrow 0$, we have $x_{ir_+}\bullet A_{ir_+,r_+r_+}(c_{r_+}) = 0$.
Since $x_{ir_+}\in \bbE_{ir_+}$ is arbitrary, by \eqref{enum:element_0} of Lemma~\ref{lem:submatrix_zero}, we obtain $A_{ir_+,r_+r_+} = 0$.

\noindent
\fbox{Case~\ref{enum:case_b_rkg}}
For any $x_{ir_+} \in \bbE_{ir_+}$, $x_{kr_+} \in \bbE_{kr_+}$, and $\epsilon > 0$, let
\begin{align*}
x(\epsilon) &\coloneqq (\epsilon^2 c_i \pm \epsilon x_{ir_+} + c_{r_+})^2 + (\epsilon^2 c_k + \epsilon x_{k r_+} + c_{r_+})^2\\
&= \underbrace{\epsilon^4c_i + \epsilon^2c_i\circ x_{ir_+}^2}_
{\in \bbE_{ii}} + \underbrace{\epsilon^4c_k + \epsilon^2c_k\circ x_{kr_+}^2}_{\in \bbE_{kk}} \pm \underbrace{(\epsilon^3 + \epsilon)x_{ir_+}}_{\in \bbE_{ir_+}}\\
&\quad + \underbrace{(\epsilon^3 + \epsilon)x_{kr_+}}_{\in \bbE_{kr_+}} + \underbrace{2c_{r_+} + \epsilon^2c_{r_+}\circ x_{ir_+}^2 + \epsilon^2 c_{r_+}\circ x_{kr_+}^2}_{\in \bbE_{r_+r_+}} \in \bbE_+.
\end{align*}
Then it follows that
\begin{align}
0 &\le q_A(x(\epsilon)) \nonumber\\
&= x(\epsilon)\bullet\!\!\!\bordermatrix{
& ii & kk & ir_+ & kr_+ & r_+r_+ \cr
& A_{ii,ii} & A_{ii,kk} & A_{ii,ir_+}  & A_{ii,kr_+} & \eqref{eq:Aijr+r+_eq_0} \cr
&  & A_{kk,kk} & A_{kk,ir_+} & A_{kk,kr_+} & \eqref{eq:Aijr+r+_eq_0} \cr
&  &  & \eqref{eq:Air+ir+_eq_0} & A_{ir_+,kr_+} & \eqref{enum:case_a_rkg} \cr
&  &  &  & \eqref{eq:Air+ir+_eq_0} & \eqref{enum:case_a_rkg} \cr
&  &  &  & & \eqref{eq:Ar+r+r+r+_eq_0}}(x(\epsilon)) \label{eq:case_b_rkg_qf_mat}\\
&= \pm 2\epsilon^2(\epsilon^2 + 1)^2 x_{ir_+}\bullet A_{ir_+,kr_+}(x_{kr_+}) + O(\epsilon^3). \label{eq:case2_rkg_qf}
\end{align}
Dividing \eqref{eq:case2_rkg_qf} by $2\epsilon^2(\epsilon^2 + 1)^2$ and letting $\epsilon \downarrow 0$, we have $x_{ir_+}\bullet A_{ir_+,kr_+}(x_{kr_+}) = 0$.
Since $x_{ir_+} \in \bbE_{ir_+}$ and $x_{kr_+} \in \bbE_{kr_+}$ are arbitrary, we obtain $A_{ir_+,kr_+} = 0$.

\noindent
\fbox{Case~\ref{enum:case_c_rkg}}
For any $x_{ir_+} \in \bbE_{ir_+}$ and $\epsilon > 0$, let
\begin{align*}
x(\epsilon) &\coloneqq \epsilon^3c_i + (\pm\epsilon^2x_{ir_+} + c_{r_+})^2 \\
&= \underbrace{\epsilon^3c_i + \epsilon^4c_i\circ x_{ir_+}^2}_{\in \bbE_{ii}} \pm \underbrace{\epsilon^2x_{ir_+}}_{\in \bbE_{ir_+}} + \underbrace{c_{r_+} + \epsilon^4c_{r_+}\circ x_{ir_+}^2}_{\in \bbE_{r_+r_+}} \in \bbE_+.
\end{align*}
Then it follows that
\begin{align}
0 &\le q_A(x(\epsilon)) \nonumber\\
&= x(\epsilon) \bullet\!\!\!\bordermatrix{
& ii & ir_+ & r_+r_+ \cr
& A_{ii,ii} & A_{ii,ir_+} & \eqref{eq:Aijr+r+_eq_0} \cr
&  & \eqref{eq:Air+ir+_eq_0} & \eqref{enum:case_a_rkg} \cr
&  &  & \eqref{eq:Ar+r+r+r+_eq_0}
}(x(\epsilon)) \nonumber\\
&= \pm 2\epsilon^5 c_i\bullet A_{ii,ir_+}(x_{ir_+}) + O(\epsilon^6). \label{eq:case3_rkg_qf}
\end{align}
Dividing \eqref{eq:case3_rkg_qf} by $2\epsilon^5$ and letting $\epsilon \downarrow 0$, we have $c_i\bullet A_{ii,ir_+}(x_{ir_+}) = 0$.
Since $x_{ir_+} \in \bbE_{ir_+}$ is arbitrary, we obtain $A_{ii,ir_+} = 0$.

\noindent
\fbox{Case~\ref{enum:case_d_rkg}}
For any $x_{kr_+} \in \bbE_{kr_+}$ and $\epsilon > 0$, let
\begin{align*}
x(\epsilon) &\coloneqq \epsilon^2c_i + (\pm \epsilon x_{kr_+} +c_{r_+})^2 \\
&= \underbrace{\epsilon^2c_i}_{\in \bbE_{ii}} + \underbrace{\epsilon^2c_k\circ x_{kr_+}^2}_{\in \bbE_{kk}} \pm \underbrace{\epsilon x_{kr_+}}_{\in \bbE_{kr_+}} + \underbrace{c_{r_+} + \epsilon^2c_{r_+}\circ x_{kr_+}^2}_{\in \bbE_{r_+r_+}} \in \bbE_+.
\end{align*}
Then it follows that
\begin{align}
0 &\le q_A(x(\epsilon)) \nonumber \\
&= x(\epsilon)\bullet\!\!\!\bordermatrix{
& ii & kk & kr_+ & r_+r_+ \cr
& A_{ii,ii} & A_{ii,kk} & A_{ii,kr_+} & \eqref{eq:Aijr+r+_eq_0} \cr
&  & A_{kk,kk} & \eqref{enum:case_c_rkg} & \eqref{eq:Aijr+r+_eq_0} \cr
&  &  & \eqref{eq:Air+ir+_eq_0} & \eqref{enum:case_a_rkg} \cr
&  &  &   & \eqref{eq:Ar+r+r+r+_eq_0}}(x(\epsilon)) \nonumber\\
&= \pm 2\epsilon^3 c_i\bullet A_{ii,kr_+}(x_{kr_+}) + O(\epsilon^4). \label{eq:case4_rkg_qf}
\end{align}
Dividing \eqref{eq:case4_rkg_qf} by $2\epsilon^3$ and letting $\epsilon \downarrow 0$, we have $c_i\bullet A_{ii,kr_+}(x_{kr_+}) = 0$.
Since $x_{kr_+} \in \bbE_{kr_+}$ is arbitrary, we obtain $A_{ii,kr_+} = 0$.

\noindent
\fbox{Case~\ref{enum:case_e_rkg}}
Here, we restrict attention to the case $k = i$ and show that $A_{ij,ir_+} = 0$.
The equality $A_{ij,jr_+} = 0$ can be proved analogously.
For any $x_{ij} \in \bbE_{ij}$, $x_{ir_+} \in \bbE_{ir_+}$, and $\epsilon > 0$, let
\begin{align*}
x(\epsilon) &\coloneqq \epsilon^2(c_i \pm x_{ij})^2 + (\epsilon x_{ir_+} + c_{r_+})^2\\
&= \underbrace{\epsilon^2c_i + \epsilon^2c_i\circ x_{ij}^2 + \epsilon^2c_i\circ x_{ir_+}^2}_{\in \bbE_{ii}} \pm \underbrace{\epsilon^2 x_{ij}}_{\in \bbE_{ij}}\\
&\quad + \underbrace{\epsilon^2c_j\circ x_{ij}^2}_{\in \bbE_{jj}} + \underbrace{\epsilon x_{ir_+}}_{\in \bbE_{ir_+}} + \underbrace{c_{r_+} + \epsilon^2c_{r_+}\circ x_{ir_+}^2}_{\in \bbE_{r_+r_+}} \in \bbE_+.
\end{align*}
Then it follows that
\begin{align}
0 &\le q_A(x(\epsilon)) \nonumber\\
&= x(\epsilon)\bullet\!\!\!\bordermatrix{
& ii & ij & jj & ir_+ & r_+r_+ \cr
& A_{ii,ii} & \eqref{eq:A_eq_aH_on_Ec1} & A_{ii,jj}  & \eqref{enum:case_c_rkg} & \eqref{eq:Aijr+r+_eq_0} \cr
&  & \eqref{eq:A_eq_aH_on_Ec1} & \eqref{eq:A_eq_aH_on_Ec1} & A_{ij,ir_+} & \eqref{eq:Aijr+r+_eq_0} \cr
&  &  & A_{jj,jj} & \eqref{enum:case_d_rkg} & \eqref{eq:Aijr+r+_eq_0} \cr
&  &  &  & \eqref{eq:Air+ir+_eq_0} & \eqref{enum:case_a_rkg} \cr
&  &  &  & & \eqref{eq:Ar+r+r+r+_eq_0}}(x(\epsilon)) \nonumber\\
&= \pm 2\epsilon^3 x_{ij}\bullet A_{ij,ir_+}(x_{ir_+}) + O(\epsilon^4). \label{eq:case5_rkg_qf}
\end{align}
Dividing \eqref{eq:case5_rkg_qf} by $2\epsilon^3$ and letting $\epsilon \downarrow 0$, we have $x_{ij}\bullet A_{ij,ir_+}(x_{ir_+}) = 0$.
Since $x_{ij} \in \bbE_{ij}$ and $x_{ir_+} \in \bbE_{ir_+}$ are arbitrary, we obtain $A_{ij,ir_+} = 0$.

\noindent
\fbox{Case~\ref{enum:case_f_rkg}}
For any $x_{ij} \in \bbE_{ij}$, $x_{kr_+} \in \bbE_{kr_+}$, and $\epsilon > 0$, let
\begin{align*}
x(\epsilon) &\coloneqq \epsilon^2(c_i \pm x_{ij})^2 + (\epsilon x_{kr_+} + c_{r_+})^2\\
&= \underbrace{\epsilon^2c_i + \epsilon^2c_i \circ x_{ij}^2}_{\in \bbE_{ii}} \pm \underbrace{\epsilon^2x_{ij}}_{\in \bbE_{ij}} + \underbrace{\epsilon^2 c_j\circ x_{ij}^2}_{\in \bbE_{jj}}\\
&\quad + \underbrace{\epsilon^2 c_k\circ x_{kr_+}^2}_{\in \bbE_{kk}} + \underbrace{\epsilon x_{kr_+}}_{\in \bbE_{kr_+}} + \underbrace{c_{r_+} + \epsilon^2c_{r_+}\circ x_{kr_+}^2}_{\in \bbE_{r_+r_+}} \in \bbE_+.
\end{align*}
Then it follows that
\begin{align}
0 &\le q_A(x(\epsilon)) \nonumber\\
&= x(\epsilon)\bullet\!\!\!\bordermatrix{
& ii & ij & jj & kk & kr_+ & r_+r_+ \cr
& A_{ii,ii} & \eqref{eq:A_eq_aH_on_Ec1} & A_{ii,jj} & A_{ii,kk} & \eqref{enum:case_d_rkg} & \eqref{eq:Aijr+r+_eq_0} \cr
&  & \eqref{eq:A_eq_aH_on_Ec1} & \eqref{eq:A_eq_aH_on_Ec1} & \eqref{eq:A_eq_aH_on_Ec1} & A_{ij,kr_+} & \eqref{eq:Aijr+r+_eq_0} \cr
&  &  & A_{jj,jj} & A_{jj,kk} & \eqref{enum:case_d_rkg} & \eqref{eq:Aijr+r+_eq_0} \cr
&  &  &  & A_{kk,kk} & \eqref{enum:case_c_rkg} & \eqref{eq:Aijr+r+_eq_0} \cr
&  &  &  & & \eqref{eq:Air+ir+_eq_0} & \eqref{enum:case_a_rkg} \cr
&  &  &  & & & \eqref{eq:Ar+r+r+r+_eq_0}}(x(\epsilon)) \nonumber\\
&= \pm 2\epsilon^3 x_{ij}\bullet A_{ij,kr_+}(x_{kr_+}) + O(\epsilon^4). \label{eq:case6_rkg_qf}
\end{align}
Dividing \eqref{eq:case6_rkg_qf} by $2\epsilon^3$ and letting $\epsilon \downarrow 0$, we have $x_{ij}\bullet A_{ij,kr_+}(x_{kr_+}) = 0$.
Since $x_{ij} \in \bbE_{ij}$ and $x_{kr_+} \in \bbE_{kr_+}$ are arbitrary, we obtain $A_{ij,kr_+} = 0$.
\end{proof}

\begin{theorem}\label{thm:Horn_exposed_symcone}
Suppose that the rank of the symmetric cone $\bbK$ is at least $5$.
For a set $\Delta = \{\delta_1,\dots,\delta_5\}$ whose elements are pairwise orthogonal and generate extreme rays of $\bbK$, the Horn transformation $H_{\Delta}$ generates an exposed ray of $\COP(\bbK)$.
\end{theorem}

\begin{proof}
As mentioned in Section~\ref{subsec:EJA}, there exist $a_i > 0$ and a primitive idempotent $c_i$ in $\bbE$ such that $\delta_i = a_ic_i$ for each $i = 1,\dots,5$, and $C\coloneqq \{c_1,\dots,c_5\}$ is a set of orthogonal primitive idempotents in $\bbE$.
Let $r$ denote the rank of the Euclidean Jordan algebra $\bbE$.
We complement $C$ so that $\{c_1,\dots,c_r\}$ is a Jordan frame of $\bbE$.
Take positive numbers $a_6,\dots,a_r$ arbitrarily and let $c \coloneqq \sum_{i=1}^r\sqrt{a_i}c_i$.
The quadratic representation $Q_c$, which is a self-adjoint linear transformation on $\bbE$ that maps $x$ to $2c\circ(c\circ x) - c^2\circ x$, is an automorphism of $\bbK$ since $c$ is invertible~\cite[Proposition~\RNum{3}.2.2]{FK1994}.
It follows from \cite[Corollary~12]{GS2013} that the linear transformation on $\calS(\bbE)$ that maps $A$ to $Q_cAQ_c$ is an automorphism of $\COP(\bbK)$.
Combining this with Lemma~\ref{lem:Horn_exposed}, we conclude that $Q_c H_C Q_c$ generates an exposed ray of $\COP(\bbK)$.
Moreover, a direct calculation shows that $Q_c H_C Q_c = H_{\Delta}$.
This yields the desired result.
\end{proof}

\section{Non-Exactness of a Sum-of-Squares Inner-Approximation Hierarchy}\label{sec:non_exact}
Let $\bbK$ be a symmetric cone, and let $(\bbE,\circ,\bullet)$ be an associated Euclidean Jordan algebra satisfying $\bbE_+ = \bbK$.
We write $n$ for the dimension of $\bbE$.
For each nonnegative integer $l$, Nishijima and Nakata~\cite{NN2024_Approximation} defined
\begin{equation}
\calI_l(\bbE_+) \coloneqq \{A \in \calS(\bbE) \mid (x\bullet x)^l q_A(x^2) \in \Sigma^{n,2(l+2)}(\bbE)\}. \label{eq:hierarchy}
\end{equation}
Here we use the notation $\calI_l(\bbE_+)$ instead of $\calI_l(\bbK)$ because it depends on the choice of the Euclidean Jordan algebra $\bbE$ associated with $\bbK$.
They showed in \cite[Theorem~3.4]{NN2024_Approximation} that the sequence $(\calI_l(\bbE_+))_{l\ge 0}$ is an \emph{inner-approximation hierarchy for $\COP(\bbK)$}: it is non-decreasing in $l$, each $\calI_l(\bbE_+)$ is included in $\COP(\bbK)$, and the interior of $\COP(\bbK)$ is contained in $\bigcup_{l\ge 0} \calI_l(\bbE_+)$.
In particular, we have the following inclusion:
\begin{equation}
\bigcup_{l\ge 0} \calI_l(\bbE_+) \subseteq \COP(\bbK). \label{eq:hierarchy_inclusion}
\end{equation}
If the Euclidean Jordan algebra $\bbE$ is the Hadamard Euclidean Jordan algebra $\bbR^r$, the hierarchy $(\calI_l(\bbR_+^r))_{l\ge 0}$ in $\bbS^r$ reduces to that provided by Parrilo~\cite{Parrilo2000_Semidefinite}.

In this section, we demonstrate the non-exactness of the inner-approximation hierarchy $(\calI_l(\bbE_+))_{l\ge 0}$ for $\COP(\bbK)$ when the rank of the symmetric cone $\bbK$ is at least $5$.
In Section~\ref{subsec:Horn_zeroth}, we prove that the Horn transformation defined in \eqref{eq:Horn_tr} does not belong to the zeroth level $\calI_0(\bbE_+)$ of this hierarchy.
In Section~\ref{subsec:not_exact_over_psd}, we show that when the Euclidean Jordan algebra $\bbE$ is associated with the positive semidefinite cone $\bbS_+^r$ for $r\ge 5$, the inclusion in \eqref{eq:hierarchy_inclusion} is indeed strict.

\subsection{Horn Transformations and Zeroth-Level Non-Exactness}\label{subsec:Horn_zeroth}
In this subsection, for a general symmetric cone $\bbK$ of rank at least $5$, we provide a certificate that the copositive cone $\COP(\bbK)$ is distinct from the zeroth level of the hierarchy $(\calI_l(\bbE_+))_{l\ge 0}$ defined in \eqref{eq:hierarchy}.

We begin by recalling the corresponding result for Hadamard Euclidean Jordan
algebras.
In this case, the hierarchy reduces to Parrilo's hierarchy for
copositive cones over nonnegative orthants, and the Horn matrix or its embedding gives a certificate for the failure of exactness at the zeroth level~\cite[Section~5]{Parrilo2000_Semidefinite}.

\begin{theorem}\label{thm:Horn_not_sos_Hadamard}
Let $r$ be an integer at least $5$.
Let $\widetilde{\bm{H}}\in\bbS^r$ be the Horn matrix $\bm{H}$ defined in \eqref{eq:Horn_mat} if $r=5$, and its zero-padded embedding if $r>5$.
Then the polynomial $(\bm{x}\odot \bm{x})^\top\widetilde{\bm{H}} (\bm{x}\odot \bm{x})$ is not a sum of squares; that is, the matrix $\widetilde{\bm{H}}$ does not belong to $\calI_0(\bbR_+^r)$.
\end{theorem}

For every nonnegative integer $l$, the cone $\calI_l(\bbE_+)$ is a spectrahedral shadow.
This follows from \eqref{enum:image} of Lemma~\ref{lem:spec_shadow}, since $\Sigma^{n,2(l+2)}(\bbE)$ is a spectrahedral shadow \cite[Corollary~3.40]{Parrilo2012} and $\calI_l(\bbE_+)$ is its linear inverse image.
On the other hand, Theorem~\ref{thm:COP_not_shadow} shows that $\COP(\bbK)$ is not a spectrahedral shadow, so $\calI_l(\bbE_+)$ is strictly included in $\COP(\bbK)$ for every nonnegative integer $l$.

Theorem~\ref{thm:Horn_not_sos_Hadamard} gives such a certificate in the case of nonnegative orthants.
In the following theorem, we extend this certificate to general symmetric cones of rank at least $5$.

\begin{theorem}\label{thm:not_sos}
Suppose that the rank of the symmetric cone $\bbK$ is at least $5$.
Recall that $(\bbE,\circ,\bullet)$ is the associated Euclidean Jordan algebra satisfying $\bbE_+ = \bbK$.
For a set $\Delta = \{\delta_1,\dots,\delta_5\}$ whose elements are pairwise orthogonal and generate extreme rays of $\bbK$,
it follows that $q_{H_{\Delta}}(x^2) \not\in \Sigma^{n,4}(\bbE)$, i.e., $H_{\Delta} \not\in \calI_0(\bbE_+)$.
\end{theorem}

\begin{proof}
For each $i = 1,\dots,5$, we take $a_i > 0$ and a primitive idempotent $c_i$ in $\bbE$ satisfying $\delta_i = a_ic_i$.
For $\bm{t} \in \bbR^5$, we define
\begin{equation*}
x(\bm{t}) \coloneqq \sum_{i=1}^5 \frac{t_i}{\sqrt{a_i}\norm{c_i}}c_i \in \bbE.
\end{equation*}
Then $\delta_i \bullet x(\bm{t})^2 = t_i^2$ for every $i = 1,\dots,5$.
This implies that, using the notation $\bm{v}_{\Delta}(x)$ introduced in \eqref{eq:vdelta}, we have $\bm{v}_{\Delta}(x(\bm{t})^2) = \bm{t}\odot \bm{t}$.
Hence, we obtain
\begin{equation}
q_{H_{\Delta}}(x(\bm{t})^2) = (\bm{t}\odot \bm{t})^\top \bm{H}(\bm{t}\odot \bm{t}). \label{eq:qxt}
\end{equation}

Now, to derive a contradiction, we assume that $q_{H_{\Delta}}(x^2) \in \Sigma^{n,4}(\bbE)$.
Then there exist quadratic forms $q_1,\dots,q_m$ such that $q_{H_{\Delta}}(x^2) = \sum_{i=1}^m q_i^2(x)$.
By substituting $x(\bm{t})$ for $x$ in $q_{H_{\Delta}}(x^2)$, we see from \eqref{eq:qxt} that
\begin{equation}
(\bm{t}\odot \bm{t})^\top \bm{H}(\bm{t}\odot \bm{t}) = \sum_{i=1}^m q_i^2(x(\bm{t})). \label{eq:qxt_sos}
\end{equation}
The right-hand side of \eqref{eq:qxt_sos} is a sum of squares of the quadratic forms $q_1(x(\bm{t})),\dots,q_m(x(\bm{t}))$ in $\bm{t}$.
However, by Theorem~\ref{thm:Horn_not_sos_Hadamard}, the left-hand side of \eqref{eq:qxt_sos} is not a sum of squares, which is a contradiction.
Thus, we obtain the desired result.
\end{proof}

\subsection{Failure of Asymptotic Exactness over Positive Semidefinite Cones}\label{subsec:not_exact_over_psd}
In this subsection, we discuss the strictness of the inclusion shown in \eqref{eq:hierarchy_inclusion}.
When the Euclidean Jordan algebra $\bbE$ is the Hadamard Euclidean Jordan algebra $\bbR^r$, the inclusion is tight if $r\le 5$ and strict if $r\ge 6$, as discussed in Section~\ref{sec:intro}.
Although one might expect an analogous result to hold for general symmetric cones, this is not immediately the case.
Specifically, we focus on the positive semidefinite cone $\bbS_+^r$ and its associated Euclidean Jordan algebra $(\bbS^r,\circ,\bullet)$ defined by
\begin{equation}
\bm{X}\circ \bm{Y} \coloneqq \frac{\bm{X}\bm{Y} + \bm{Y}\bm{X}}{2}
\quad \text{and} \quad
\bm{X}\bullet \bm{Y} \coloneqq \tr(\bm{X}\bm{Y}), \label{eq:EJA_psd}
\end{equation}
respectively, for $\bm{X},\bm{Y}\in \bbS^r$~\cite[page~48]{FK1994}.
Note that $\bm{X}^2 = \bm{X} \circ \bm{X}$ coincides with the usual matrix square.
We show in Theorem~\ref{thm:not_exact_psd} that, for this cone, the inclusion in \eqref{eq:hierarchy_inclusion} is strict for every $r\ge 5$; in particular, strictness already occurs at $r=5$.
As will become clear from the discussion below, the proof is constructive and provides an explicit certificate for the strictness of the inclusion.

The following lemma follows directly from \cite[Theorem~4.3]{CL1977}, but we include the proof to exhibit a concrete instance satisfying its condition.

\begin{lemma}\label{lem:psd_but_nonsos}
Let $r$ be an integer at least $5$.
Then there exists a positive semidefinite quartic form $p\colon \bbR^r \to \bbR$ such that $(\sum_{i=1}^rx_i^2)^l p(\bm{x})$ is not a sum of squares for any nonnegative integer $l$.
\end{lemma}

\begin{proof}
Let $p_0(x_1,x_2,x_3,x_4) \coloneqq x_1^2x_2^2 + x_1^2x_3^2 + x_2^2x_3^2 + x_4^4 -4x_1x_2x_3x_4$.
It is known that $p_0$ is positive semidefinite but not a sum of squares~\cite[Theorem~2.3]{CL1977}.
We then define $p(\bm{x}) \coloneqq p_0(x_1,x_2,x_3,x_4)$ for $\bm{x}\in \bbR^r$.
We show hereafter that the polynomial $p$ satisfies the desired condition.
To see this, we assume that $(\sum_{i=1}^rx_i^2)^l p(\bm{x})$ is a sum of squares for some nonnegative integer $l$.
Setting $x_5 = \cdots = x_{r-1} = 0$ and $x_r = 1$, we see that the polynomial $(\sum_{i=1}^4x_i^2 + 1)^l p_0(x_1,x_2,x_3,x_4)$ is also a sum of squares.
The lowest-degree homogeneous component of this polynomial is $p_0$, which must also be a sum of squares~\cite[Lemma~4]{LV2023}, contradicting the choice of $p_0$.
\end{proof}

\begin{theorem}\label{thm:not_exact_psd}
Let $r$ be an integer at least $5$.
Then there exists $A\in \COP(\bbS_+^r)$ such that $(\bm{X}\bullet \bm{X})^lq_A(\bm{X}^2)$ is not a sum of squares for any nonnegative integer $l$.
In other words, $\bigcup_{l\ge 0} \calI_l(\bbS_+^r)$ is strictly included in $\COP(\bbS_+^r)$.
\end{theorem}

\begin{proof}
Since $r\ge 5$, we can choose a positive semidefinite quartic form $p\colon \bbR^r \to \bbR$ satisfying the condition stated in Lemma~\ref{lem:psd_but_nonsos}.
Let $q_1\colon \bbS^r \to \bbR$ be a quadratic form satisfying $q_1(\bm{x}\bm{x}^\top) = p(\bm{x})$ for all $\bm{x}\in \bbR^r$.
(Such a $q_1$ can be constructed in the following way.
Since $p$ is a quartic form, there exist a subset $I$ of $\{(h,i,j,k) \mid 1\le h\le i\le j\le k\le r\}$ and real numbers $p_{hijk}$ for $(h,i,j,k)\in I$ such that $p(\bm{x}) = \sum_{(h,i,j,k)\in I} p_{hijk}x_hx_ix_jx_k$.
Then $q_1(\bm{X}) \coloneqq \sum_{(h,i,j,k)\in I} p_{hijk}X_{hi}X_{jk}$ is a desired quadratic form.)
We also define the quadratic form $q_2$ by $q_2(\bm{X}) \coloneqq \tr(\bm{X})^2 - \tr(\bm{X}^2)$.
Let $A_1$ be an element in $\calS(\bbS^r)$ satisfying $q_1 = q_{A_1}$ and define $\alpha$ by
\begin{equation*}
\alpha \coloneqq -\min_{\norm{\bm{x}}_2 = \norm{\bm{y}}_2 = 1} \bm{x}\bm{x}^\top \bullet A_1(\bm{y}\bm{y}^\top),
\end{equation*}
where $\norm{\cdot}_2$ denotes the $2$-norm of an input vector.
Furthermore, we define the quadratic form $q$ by $q(\bm{X}) \coloneqq q_1(\bm{X}) + \alpha q_2(\bm{X})$, and let $A\in \calS(\bbS^r)$ satisfy $q = q_A$.
In what follows, we show that $A$ is a desired element.

First, we show that $A\in \COP(\bbS_+^r)$.
Let $\bm{X}\in \bbS_+^r$ be arbitrary and let $\sum_{i=1}^r \lambda_i \bm{u}_i\bm{u}_i^\top$ be an eigenvalue decomposition of $\bm{X}$, where $\lambda_1,\dots,\lambda_r$ are the eigenvalues of $\bm{X}$ and $\bm{u}_1,\dots,\bm{u}_r$ are orthonormal.
Note that the eigenvalues $\lambda_1,\dots,\lambda_r$ are nonnegative since $\bm{X}$ is positive semidefinite.
Then it follows that
\begin{align*}
q_A(\bm{X}) &= q_1(\bm{X}) + \alpha q_2(\bm{X}) \\
&= \sum_{i=1}^r\lambda_i^2 q_1(\bm{u}_i\bm{u}_i^\top) + 2\sum_{1\le i<j \le r}\lambda_i\lambda_j\bm{u}_i\bm{u}_i^\top \bullet A_1(\bm{u}_j\bm{u}_j^\top) + 2\alpha\sum_{1\le i<j\le r}\lambda_i\lambda_j\\
&\overset{\scriptsize \text{(a)}}\ge \sum_{i=1}^r\lambda_i^2 q_1(\bm{u}_i\bm{u}_i^\top)\\
&\overset{\scriptsize \text{(b)}}= \sum_{i=1}^r\lambda_i^2p(\bm{u}_i)\\
&\overset{\scriptsize \text{(c)}}\ge 0,
\end{align*}
where (a) follows from the definition of $\alpha$, (b) follows from the choice of $q_1$, and (c) follows from the positive semidefiniteness of $p$.
Since $\bm{X} \in \bbS_+^r$ is arbitrary, we obtain $A\in \COP(\bbS_+^r)$.

Second, we show that $A \not\in \bigcup_{l\ge 0} \calI_l(\bbS_+^r)$, i.e., $(\bm{X}\bullet \bm{X})^l q_A(\bm{X}^2)$ is not a sum of squares for any nonnegative integer $l$.
To derive a contradiction, we assume that $(\bm{X}\bullet \bm{X})^l q_A(\bm{X}^2)$ is a sum of squares for some nonnegative integer $l$.
Substituting $\bm{x}\bm{x}^\top$ for $\bm{X}$ in this polynomial, we have
\begin{align}
(\bm{x}\bm{x}^\top\bullet \bm{x}\bm{x}^\top)^l q_A((\bm{x}\bm{x}^\top)^2) &= \left(\sum_{i=1}^r x_i^2\right)^{2l+2}q_A(\bm{x}\bm{x}^\top) \nonumber \\
&= \left(\sum_{i=1}^r x_i^2\right)^{2l+2}q_1(\bm{x}\bm{x}^\top) \nonumber \\
&= \left(\sum_{i=1}^r x_i^2\right)^{2l+2}p(\bm{x}), \label{eq:XdotXl_qX2}
\end{align}
where the second equality follows from the fact that $q_2$ vanishes on matrices of rank at most $1$, i.e., $q_2(\bm{x}\bm{x}^\top) = 0$, and the third equality follows from the choice of $q_1$.
By the assumption for contradiction, the polynomial in \eqref{eq:XdotXl_qX2} is also a sum of squares.
However, this contradicts the condition for the polynomial $p$ stated in Lemma~\ref{lem:psd_but_nonsos}.
Thus, we obtain $A \not\in \bigcup_{l\ge 0} \calI_l(\bbS_+^r)$.
\end{proof}

The contrast between the Hadamard Euclidean Jordan algebra $\bbR^5$ and the algebra $\bbS^5$ in \eqref{eq:EJA_psd} comes from the different polynomial structures induced by the squaring mapping.
In the case of the Hadamard Euclidean Jordan algebra, we have $\bm{x}^2=\bm{x}\odot\bm{x}$, so for a matrix $\bm{A}\in\bbS^5$, the polynomial $(\bm{x}^2)^\top\bm{A}\bm{x}^2$ is a quartic form in which each variable appears only through its square.
In the case of the matrix algebra $\bbS^5$, however, restricting a quadratic form on $\bbS^5$ to rank-one matrices $\bm{X}=\bm{x}\bm{x}^\top$ gives arbitrary quartic forms, since $X_{ij}=x_i x_j$ for each $1\le i\le j\le 5$.
Hence the hierarchy over $\bbS_+^5$ contains, as a restriction, the sum-of-squares representability problem for general positive semidefinite quartic forms on $\bbR^5$ after multiplication by a power of $\sum_{i=1}^5 x_i^2$.
This additional freedom allows one to construct elements in $\COP(\bbS_+^5)$ outside $\bigcup_{l\ge0}\calI_l(\bbS_+^5)$.

An argument similar to Theorem~\ref{thm:not_exact_psd} does not apply to the case $r\le 4$.
At present, Lemma~\ref{lem:psd_but_nonsos} is guaranteed to hold only for $r\ge 5$, and it remains unknown whether an analogous result holds for $r=4$.
For $r \le 3$, a result analogous to this lemma does not hold, since every positive semidefinite quartic form in at most $3$ variables is a sum of squares~\cite{Hilbert1888}.

\section{Final Remarks}\label{sec:remark}
In this paper, we showed that for any symmetric cone of rank $r$ at least $5$, neither $\COP(\bbK)$ nor $\CP(\bbK)$ is a spectrahedral shadow.
We proved this result by taking a slice of $\COP(\bbK)$ that is linearly isomorphic to $\COP^r$.

When $r$ is less than or equal to $4$, we cannot use this technique to solve the problem of whether $\COP(\bbK)$ and $\CP(\bbK)$ are spectrahedral shadows, since $\COP^r$ and $\CP^r$ are spectrahedral shadows as mentioned in Section~\ref{sec:intro}.
We note that $\COP(\bbK)$ and $\CP(\bbK)$ are spectrahedral shadows in the case where $r = 1$ and $r = 2$.
When $r = 1$, both $\COP(\bbK)$ and $\CP(\bbK)$ are $1$-dimensional polyhedral cones, which are spectrahedral shadows.
When $r = 2$, according to \cite[Appendix~A]{NL2026}, $\COP(\bbK)$ and $\CP(\bbK)$ are linearly isomorphic to copositive and completely positive cones over a second-order cone, respectively, and hence are spectrahedral shadows, as mentioned in Section~\ref{sec:intro}.
The only remaining cases are $r = 3$ and $r = 4$.
The following question remains open:
\begin{question}
For a symmetric cone of rank $3$ or $4$, are the copositive and completely positive cones over it spectrahedral shadows?
\end{question}

\vspace{0.5cm}
\noindent
{\bf Acknowledgments}
The author would like to thank Professor Bruno F. Louren\c{c}o for his valuable comments and discussions.
The author also thanks the anonymous reviewers for their constructive comments, which improved the quality of this paper.
The author is supported by JSPS Grant-in-Aid for Research Activity Start-up JP25K23344.

\begin{appendices}
\section*{Appendix: On the Case Distinction in the Proof of Lemma~\ref{lem:Horn_exposed}}
In this appendix, we justify the case distinction from Case~\ref{enum:case_a_rkg} to Case~\ref{enum:case_f_rkg} in the proof of Lemma~\ref{lem:Horn_exposed}.
First, we observe that unless the quadruple $(i,j,k,l)$ with $(ij,kl) \in \llbracket r_+\rrbracket$ satisfies $(ij,kl) \in \llbracket r\rrbracket$, which corresponds to \eqref{eq:A_eq_aH_on_Ec1}, it follows that $l = r_+$.
In what follows, we therefore focus on the case $l = r_+$.

First, we consider the case $k = r_+$.
If $j \le r$, then the quadruple $(i,j,k,l)$ corresponds to \eqref{eq:Aijr+r+_eq_0}.
If $j = r_+$, we distinguish two subcases: $1 \le i \le r$ and $i = r_+$.
The case $1 \le i \le r$ corresponds to Case~\ref{enum:case_a_rkg}, whereas the case $i = r_+$ corresponds to \eqref{eq:Ar+r+r+r+_eq_0}.

Next, we consider the case $k \le r$.
When $j = r_+$, we distinguish two subcases: $i = k$ and $i < k$.
The case $i = k$ corresponds to \eqref{eq:Air+ir+_eq_0}, whereas the case $i < k$ corresponds to Case~\ref{enum:case_b_rkg}.
When $j \le r$, we distinguish subcases according to how many of $i,j,k,l$ are identical.
Since $l = r_+$, it is impossible that all of $i,j,k,l$ are identical.
If exactly three of $i,j,k,l$ are identical, they must be $i,j,k$, which corresponds to Case~\ref{enum:case_c_rkg}.
If exactly two of $i,j,k,l$ are identical, there are two subcases.
The case $i = j$ corresponds to Case~\ref{enum:case_d_rkg}, and the case $i < j$ corresponds to Case~\ref{enum:case_e_rkg}.
Finally, the case where $i,j,k,l$ are all distinct corresponds to Case~\ref{enum:case_f_rkg}.
\end{appendices}

\bibliographystyle{plainurl} 
\bibliography{2025_1_ref} %
\end{document}